\documentclass[11pt]{amsart}
\usepackage{graphicx,amssymb,amsmath, latexsym,cite, amscd,stmaryrd}
\usepackage{amsthm,amsrefs}
\usepackage{enumerate,lscape}  
\usepackage{float,xcolor}
\usepackage{fancyhdr, lastpage}  
\usepackage{libertine}         
 \usepackage{url}   
\usepackage{verbatim} 
\usepackage{algorithm}  
\usepackage[noend]{algpseudocode}
    
\algblock[TryCatchFinally]{try}{endtry}
\algcblock[TryCatchFinally]{TryCatchFinally}{finally}{endtry} 
\algcblockdefx[TryCatchFinally]{TryCatchFinally}{catch}{endtry}
	[1]{\textbf{catch} #1}
	{\textbf{end try}}

\newcommand{\code}[1]{\textsf{#1}}
\usepackage[all,cmtip]{xy}

\newcommand{\mydes}[3]{  

\smallskip
 
\noindent\begin{tabular}{@{}p{0.05cm}lp{13cm}}  
\multicolumn{3}{l}{\sf #1}\\ 
& {\sf Given:\hspace*{-0.2cm}}&#2 \\
& {\sf Return:\hspace*{-0.2cm}}&#3
\end{tabular}

\medskip

}
 
\usepackage{algorithm}
\usepackage[noend]{algpseudocode}

\algblock[TryCatchFinally]{try}{endtry}
\algcblock[TryCatchFinally]{TryCatchFinally}{finally}{endtry}
\algcblockdefx[TryCatchFinally]{TryCatchFinally}{catch}{endtry}
	[1]{\textbf{exception:} #1}
	{\textbf{end try}}

\newtheoremstyle{mythm}                   
{6pt}
{6pt}
{\it}
{}
{\bf}
{.}
{.5em}
{}

\newtheoremstyle{mydef}                   
{6pt}
{6pt}
{}
{}
{\bf}
{.}
{.5em}
{}

\newtheoremstyle{myrem}                   
{6pt}
{6pt}
{}
{}
{\bf}
{.}
{.5em}
{}

\theoremstyle{mythm}      
\newtheorem{theorem}{Theorem}[section]
\newtheorem{proposition}[theorem]{Proposition}
\newtheorem{lemma}[theorem]{Lemma}
\newtheorem{corollary}[theorem]{Corollary}

\theoremstyle{mydef}      
\newtheorem{definition}[theorem]{Definition}
\newtheorem{example}[theorem]{Example}

\theoremstyle{myrem}
\newtheorem{remark}[theorem]{Remark}
\numberwithin{equation}{section}

\newcounter{ithmcount}

\newenvironment{items}{
\begin{list}{$\alph{item})$}
{\labelwidth18pt \leftmargin20pt \topsep3pt \itemsep5pt \parsep0pt}}
{\end{list}}

\allowdisplaybreaks

\reversemarginpar
\subjclass[2000]{}

\newcommand{\Aut}{{\rm Aut}}

\newcommand{\GL}{{\rm GL}}

\renewcommand{\leq}{\leqslant}

\begin{document}
 
\vspace*{-0.6cm}
   
\title{Polynomial time isomorphism tests of black-box type groups of most orders}
\subjclass[2000]{}
\author[H. Dietrich]{Heiko Dietrich}
\address{School of Mathematics, Monash University, 
Clayton VIC 3800, Australia}
\email{Heiko,Dietrich@Monash.Edu}
\author[J.B. Wilson]{James B.\ Wilson}
\address{Department of Mathematics, Colorado Sate University, Fort Collins Colorado, 80523, USA}
\email{James.Wilson@ColoState.Edu}

\keywords{finite groups, abelian groups, meta-cyclic groups, group isomorphisms}
\date{\today} 

\begin{abstract}~\\ 
{\color{blue}{\bf Important comment:} This draft contains some erroneous statements and some proofs have gaps;  please refer to the following publications for corrections:\\
    $\bullet$ {\it Isomorphism testing of groups of cube-free order.} J.\ Algebra 545 (2020) 174-197.\\
    $\bullet$ {\it Group isomorphism is nearly-linear time for most orders.} FOCS 2021 (accepted, arXiv:2011.03133)\\ 
    $\bullet $ The content on \emph{groups of black-box type} will be revised in a different draft.\\}
    
\noindent We consider the isomorphism problem for finite  abelian  groups and  finite meta-cyclic groups. We prove
that for a dense set of positive integers $n$, isomorphism testing for
abelian groups of black-box type of order $n$ can be done in
time polynomial in $\log n$. We also prove that for a dense set of orders $n$
with given prime factors, one can test isomorphism for coprime meta-cyclic
groups of black-box type of order $n$  in time polynomial in $\log n$. 
Prior methods for these two classes of groups have running times exponential in $\log n$.\\

\end{abstract}

\thanks{This research was partially supported by the Simons Foundation, the Mathematisches Forschungsinstitut Oberwolfach, and  NSF grant DMS-1620454. 
The first author thanks the Lehrstuhl D of the RWTH Aachen for the great hospitality during his Simon Visiting Professorship in Summer 2016. Both authors thank Jeff Achter for discussion on integer proportions and
Takunari Miyazaki for help with black-box groups.}

\maketitle


\vspace*{-1cm}


\section{Introduction}

Groups are one of the most prominent algebraic 
structures in science since they capture the natural concept of symmetry. Yet, it is still a difficult problem to decide whether two finite groups are isomorphic. Despite abundant knowledge about groups, presently no one has provided an isomorphism test for all finite groups 
whose complexity  improves substantively over brute-force (see \cite{Grochow}).  In the most general form,
there is no known polynomial-time isomorphism test even for non-deterministic Turing machines, that is,
the problem may lie outside the complexity classes NP and co-NP (see \cite{BS84}*{Corollary~4.9}).
At the time of this writing, the available implementations of algorithms that test 
isomorphism on broad classes of groups can run out of memory or run for days on examples 
of orders only a few thousand, see \cite{BMW}*{Section~1.1} and \cite{PART2}*{Table~1}.  
For comparison, isomorphism testing of general graphs on as 
many vertices can be performed in seconds on an average computer \cite{Traces}.  To 
isolate the critical difficulties in group isomorphism it often helps to consider special 
classes of groups, see \citelist{\cite{BQ:tower}\cite{BCGQ}\cite{BMW}\cite{cf}\cite{Wilson:profile}} for recent work.

In this paper we investigate efficient isomorphism tests suitable for abelian and meta-cyclic groups of most finite 
orders -- not to be confused with most finite groups. Recently Kavitha \cite{kavitha} proved that isomorphism testing of abelian groups of order $n$ is $O(n)$, that is, in linear time if the input is the Cayley table of the group. This is rarely the case in practical applications, so here we work with groups of \emph{black-box type} where by assumption we only 
know how to multiply, invert, test equality, and generate the group. The input size of a black box group of order $n$ can be made $O(\log n)$, so such groups can be exponentially larger than the data it takes to specify the group.  It is that expressive 
power that makes this model so useful.  However, it is exceedingly hard to learn anything about
these groups let alone decide isomorphism.  Iliopoulos \cite{Iliopoulos} demonstrated that 
most questions for abelian groups of black-box type are at least as hard as the discrete 
logarithm problem and integer factorization (neither of those problems seems to have  an efficient 
deterministic, randomized, or reliable heuristic solution).  Nevertheless, we prove that 
isomorphism testing of abelian groups is efficient for almost all group orders.  Based on 
work of Karagiorgos-Poulakis \cite{KP} and enumeration results of Erd\H{o}s-P\'{a}lfy, we 
prove the following theorem in Section \ref{secBBAb}. Recall that  a set $\mathcal{D}$ of integers 
is \emph{dense} if $\lim_{n\to\infty} |\mathcal{D}\cap\{1,2,\ldots,n\}|/n =1$.  
 
\begin{theorem}\label{thm:iso-abelian-BB}   
There is a dense set of integers $\mathcal{D}$ and a deterministic isomorphism test for 
abelian groups of black-box type with known order $n\in\mathcal{D}$ that runs in time 
polynomial in the input size.
\end{theorem}

 The set $\mathcal{D}$ in Theorem \ref{thm:iso-abelian-BB} consists of all positive integers $n$ which can be factorised as $n=p_1^{e_1}\ldots p_k^{e_k}b$, 
where $b$ is square-free, every prime divisor of $b$ is bigger than $\log\log n$, the 
primes $p_1,\ldots,p_k\leq \log\log n$ are all distinct and coprime to $b$, and each 
$p_i^{e_i}\leq \log n$, see Definition \ref{defKBfree} below. We note that $|\mathcal{D}\cap \{1,2,\ldots,10^k\}|/10^k$ for $k=3,\ldots,8$ is approximately  0.703, 0.757, 0.816, 0.822, 0.826, and 0.848, respectively.

Moving away from abelian groups the complications grow quickly, requiring more assumptions.
We consider meta-cyclic groups, that is, cyclic extensions of cyclic groups, and prove the 
following theorem in Section \ref{secBBmc}. Recall that a \emph{Las Vegas} algorithm is a 
randomized algorithm that returns an answer, but may fail to return within a user specified
probability of $\varepsilon>0$, see \cite{Seress}*{p.\ 14} for a discussion on randomised algorithms. We say a group $G$ is \emph{coprime meta-cyclic} if $G=U\ltimes K$ for cyclic subgroups $U,K\leq G$ of coprime order.

\begin{theorem}\label{thm:iso-meta-cyclic-BB}  
There is a dense set of integers $\hat{\mathcal{D}}$ and a poly\-nomial-time 
Las Vegas algorithm to test if a solvable group of black-box type and known factored 
order $n\in \hat{\mathcal{D}}$ is coprime meta-cyclic, and  to decide isomorphism of those groups.
\end{theorem}

The set $\hat{\mathcal{D}}$ will be defined as a subset of $\mathcal{D}$, see Lemma \ref{lem:Count-3}; for $k=3,\ldots,8$, we have that the quotient $|\hat{\mathcal{D}}\cap \{1,2,\ldots,10^k\}|/10^k$ is approximately  0.552, 0.669, 0.733, 0.719, 0.713,  and 0.721, respectively.

We note that solvability of finite black-box type groups can be tested by a Monte Carlo algorithm, see \cite{Seress}*{p.\ 38}. The mechanics of Theorems~\ref{thm:iso-abelian-BB} \& \ref{thm:iso-meta-cyclic-BB} 
depend in part on how the integers in the sets $\mathcal{D}$ and $\hat{\mathcal{D}}$ 
limit the possible group theory of groups of these orders. In particular, we need control
of large primes and large powers of a prime.  An integer $n$ is {\em  $k$-free} 
if no prime to the power $k$ divides $n$;  {\em square-free} and {\em cube-free} are synonymous 
with $2$-free and $3$-free. The orders for which Theorems~\ref{thm:iso-abelian-BB}~\&~\ref{thm:iso-meta-cyclic-BB} apply can be described as ``eventually square-free'', and are similar to orders studied by Erd\H{o}s \& P\'alfy~\cite{ErdosPalfy}.

For the proof of Theorem~\ref{thm:iso-meta-cyclic-BB} we show that a solvable group $G$ of order 
$n\in\hat{\mathcal{D}}$ decomposes as a semidirect product $G=K\ltimes_{\theta} \mathbb{Z}/b$ 
where $b$ is square-free, $\gcd(|K|,b)=1$, and all the \emph{large} prime divisors $p$ of $n$ with 
$p\geq \log\log n$ divide $b$, see Theorem~\ref{thm:mostly-split}. 
This reduces the problem to studying $K$ and finding $\theta$; in other words, our 
approach is to isolate all the large primes in the order of $G$ to a single subgroup. 
This should be useful when working with groups of black-box type for any purpose, 
not just isomorphism testing.

\begin{remark}
The input size of a black-box type group can be reduced to $(\log n)^{O(1)}$ in so-called Monte
Carlo polynomial time, see \cite{Seress}*{Lemma~2.3.4}.  Once this has been applied to inputs, 
Theorems~\ref{thm:iso-abelian-BB}~\&~\ref{thm:iso-meta-cyclic-BB}
are $(\log n)^{O(1)}$-time algorithms.  Prior complexity bounds for these problems were $O(\sqrt{n})$  and $O(n^{4+o(1)})$, respectively, see Remark \ref{remTeske}  and Section \ref{sec14}.
\end{remark}

By a theorem of H\"older (\cite{rob}*{10.1.10}), all groups of square-free order $n$ are meta-cyclic, but
Theorem~\ref{thm:iso-meta-cyclic-BB} is not guaranteed for all square-free orders $n$. 
However, by switching to a more restrictive computational model, we recently made progress for isomorphism
testing of cube-free groups.  Specifically, considering groups generated by a set $S$ 
of permutations on a finite set $\Omega$  gives  access to a robust family of 
algorithms by Sims and many others (see \citelist{\cite{handbook}\cite{Seress}})
that run in time polynomial in $|\Omega|\cdot |S|$. Note that the order of such a group 
$G$ can be exponential in $|\Omega|\cdot |S|$, even when restricted to groups of square-free order, see \cite{PART2}*{Proposition 2.1}.  We proved the following in \cite{PART2}.

\begin{theorem}[\cite{PART2}]
There is a polynomial-time algorithm that given groups $G$ and $H$ of permutations on
finitely many points, decides whether they are of cube-free order, and if so, decides that $G\not\cong H$ or constructs an isomorphism $G\to H$.
\end{theorem}

Completing our tour, we consider groups input by their multiplication table ({\em Cayley table}).  
Such a verbose input for groups allows us to compare the complexity of group isomorphism 
with other algebraic structures -- such as semigroups -- which in general cannot be input 
by anything smaller than a multiplication table.  The complexity of isomorphism testing 
for groups of order $n$ given by Cayley tables is $n^{(\log n)/4 + O(1)}$, see \cite{Grochow},
and a recent break-through by Babai \cite{Babai:quasi} shows that isomorphism testing of 
semigroups (equivalently graph isomorphism) has complexity
$n^{O((\log n)^d))}$ for some $d\leq 3$.
Using a theorem of Guralnick \cite{guralnick}*{Theorem A} that depends on the Classification of Finite Simple Groups (CFSG), we prove the following in Section \ref{sec14}.

\begin{theorem}[(CFSG)]\label{thm:iso-almostall-Cayley}
For every $\varepsilon>0$ there a set $\tilde{\mathcal{D}}\subset\mathbb{N}$ of density $1-\varepsilon$
such that isomorphism of groups of order $n\in \tilde{\mathcal{D}}$ 
input by Cayley tables can be decided in deterministic polynomial time.
\end{theorem}

Throughout this paper we mostly adhere to protocol set out in standard literature on computational group theory, such as the Handbook of Computation Group Theory \cite{handbook} and the books of Robinson \cite{rob} and Seress \cite{Seress}. Section~\ref{sec:prelim} provides further details on our computational assumptions; in particular, in Section~\ref{sec:prelim} we introduce the background necessary for formulating our results in the language of Type Theory. A further justification of the use of Type Theory is given in Section \ref{sec:Turing}.

\subsection{Limitations}   
While we provide isomorphism tests for groups of cube-free orders (see \cite{PART2}), and for abelian and 
meta-cyclic groups of almost all  orders, it is known that most isomorphism types of groups 
accumulate at orders with large prime-power divisors.  Indeed,  Higman, Sims, and Pyber 
\cite{BNV:enum} proved that the number of groups of order $n$, up to isomorphism, tends to 
$n^{2\mu(n)^2/27+O(\log n)}$ where $\mu(n)=\max\{k : n \text{ is not $k$-free}\}$. 
Specifically, the number of pairwise non-isomorphic groups of a cube-free order $n$ is 
not more than $O(n^8)$, with speculation that the tight bound is $o(n^2)$, see \cite{BNV:enum}*{p.~236}. 
The prevailing belief in works like \citelist{\cite{BCGQ}\cite{Wilson:profile}}
is that the difficult instances of group isomorphism
are when $\mu(n)$ is unbounded, especially when $n$ is a prime power. 
Isomorphism testing of finite $p$-groups is indeed a research area that has attracted a lot of attention.

Avoiding the problems discovered by Iliopoulos \cite{Iliopoulos} comes at a price. 
According to \cite{Adleman}, we know of no methods to test if an integer $n$ is square-free or $k$-free
without factoring $n$.  Moreover,  some groups have unknown order.  So we cannot
apply some of our algorithms in those cases. We also stress that Theorems~\ref{thm:iso-abelian-BB}~\&~\ref{thm:iso-meta-cyclic-BB} report
existence of isomorphisms only:  this is because we introduce a third group $G_0$ 
with favorable computational properties, and construct
isomorphisms $G_0\to G_1$ and $G_0\to G_2$ {\em without their inverses} -- which would require solutions to discrete logarithm
and integer factorization problems.  Thus, 
$G_1\cong G_2$ is inferred with no explicit isomorphism. Even so, having a preferred
copy $G_0$ of a group can be helpful, and, in practice, sometimes this can be used to construct an isomorphism $G_1\to G_2$.

\section{Preliminaries}\label{sec:prelim}
\subsection{Notation.} We reserve $p$ for prime numbers and $n$ for group orders.  
For a positive integer $n$ we denote by $C_n$ a
cyclic group of order $n$, and $\mathbb{Z}/n$ for the explicit encoding as 
integers, in which we are further permitted to treat the structure as a
ring.  Let $(\mathbb{Z}/n)^{\times}$ denote the units of this ring.
Direct products of groups are denoted variously by
``$\times$'' or exponents. Throughout, $\mathbb{F}_{q}$ is a field of order $q$ 
and ${\rm GL}_d(q)$ is the group of invertible $(d\times d)$-
matrices over $\mathbb{F}_{q}$. 

For a group $G$ and  $g,h\in G$, conjugates and commutators are $g^h=h^{-1}g h$ and 
$[g,h]=g^{-1} g^h$, respectively. For subsets $X,Y\subset G$ let
$[X,Y]=\langle [x,y]:x\in X, y\in Y\rangle$; the centralizer and normalizer of $X$ in $G$ are  $C_G(X)=\{g\in G: [X,g]=1\}$ and $N_G(X)=\{g\in G: [X,g]\subseteq X\}$, respectively.
The derived series of $G$ has terms  $G^{(n+1)}=[G^{(n)},G^{(n)}]$ for 
$n\geq 1$, with $G^{(1)}=G$. If $G$ is abelian and $m>0$ is an integer, then
$G^{[m]}$ is the subgroup of $G$ generated by all $m$-th powers of elements in~$G$. We read group extensions from the right and use 
$A\ltimes B$ for split extensions of $B$ by $A$; we also write 
$A\ltimes_\varphi B$ to emphasize the action $\varphi\colon A\to \Aut(B)$. 
Hence, $A\ltimes B\ltimes C\ltimes D$ stands for 
$((A\ltimes B)\ltimes C)\ltimes D$, etc.


\subsection{Computation requirements}\label{sec:general-inputs}
 
The models for computations we use here are far ranging, so instead of 
discussing complexity with Turing Machines (universal computers) we 
use Type Theory (universal programming languages).  Recently these were shown
to be equivalent \cite{BetaReduction}, but the Type Theory approach is expressive,
reflects current programming, and, importantly, it avoids certain complication
with black-box groups specific to problems that are in general not known to
be in NP, such as group isomorphism; we comment on these in more detail in  Section~\ref{sec:Turing}. For a good introduction to Type Theory we refer to  Grayson \cite{grayson} and Farmer \cite{farmer}.

The following paragraphs briefly describe the concepts of Type Theory we apply in our work. For details of our various computational preliminaries we refer to 
standard books on computational group theory \citelist{\cite{handbook}\cite{Seress}} and to \cite{HoTT}*{Chapter~1}.

\medskip

{\bf Types.}
Since we compute with groups that are too large to be listed we shall {\em not}
consider the set on which a group $G$ is defined as part of the input. Type 
Theories (derivatives of Church's $\lambda$-calculus) likewise avoid sets as 
their foundation and use instead {\em types} \code{A} and their {\em terms 
(inhabitants)} \[x : \code{A}, \quad\text{saying ``$x$ of type \code{A}''}.\]Whenever 
the terms $x : \code{A}$ form a set, that set will be denoted 
$|\code{A}|=\{x : \code{A}\}$; see \cite{HoTT}*{Section~3.1}.  For
example, a type $\code{Boolean}:\equiv\{0,1\}$ has $|\code{Boolean}|=\{0,1\}$.
New types are built from old types using sums $+$ (Type Theory's version of
``\code{or}'' and  disjoint union), products $\times$ (``\code{and}'' and intersections), 
functions, $\sum_{x:\code{A}}P(x)$ (``exists''), and $\prod_{x:\code{A}}P(x)$ (``for-all''); see
\cite{HoTT}*{Table~1}.  

Types can have infinitely many terms, for example, a type \code{Int} for integers
has $|\code{Int}|=\mathbb{Z}$; cf.\ \cite{HoTT}*{Sections~1.8-1.9}. 
{\em Algebraic types} (or {\em composition types}) are schema to build types
using generic parameter types \code{A}, for example, \code{List[A]} for lists of terms
of type \code{A}, or \code{Mat[K]} for matrices of inhabitants of~\code{K}.

\medskip

{\bf Function types.} As an algorithm proceeds, it converts terms of one type 
into terms of possibly other types.  For that we need function types. As sets
are not available, a function $f:\code{A}\to\code{B}$ here means a term $f$ of 
type $\code{A}\to\code{B}$ which expresses how to transform a term $x:\code{A}$ 
into a term $y:\code{B}$ by using previously defined terms and types.  For example, 
a function squaring integers can be written as
\[\code{def }f(x:\code{Int}):\code{Int}:\equiv x^2\]
which is a modern variant of $\lambda$-calculus notation $\lambda (x:\code{A})(\cdots):
\code{A}\to\code{B}$, cf.\ \cite{HoTT}*{Section~1.2}.  Iterating with the
function type constructor, we produce terms  
$f:\code{A}\to\code{B}\to\code{C}$ that interpret 
set-functions $|\code{A}|\to (|\code{B}|\to |\code{C}|)$, or, equivalently,
$|\code{A}|\times |\code{B}|\to |\code{C}|$.

\medskip

{\bf Partial Functions.} To model functions that are not defined for every
$x:\code{A}$ we consider the type \[\code{B}^?:\equiv \code{B}+\code{Nothing}.\]
We use this to define $f:\code{A}\dashrightarrow\code{B}:\equiv\code{A}^?\to\code{B}^?$
and call the {\em domain} those $x:\code{A}$ bound to $f(x):\code{B}$. 
For example, the following represents multiplicative inverses in $\mathbb{Z}/5$:
\[\code{def }f(x:\code{Int}^?):\code{Int}^? 
	:\equiv \text{ if } x:\code{Int}\text{ and }x\neq 0\text{ then }(x^{-1}\bmod{5}); 
	\text{ else }\code{Nothing}.\]
So instead of restricting the inputs to exclude $0$, we simply output a result
of \code{Nothing}. Note that  \code{Nothing} as input is also 
mapped to \code{Nothing}, which allows partial functions to be composed.

\medskip

{\bf Group Types.}
A type for groups is made by combining types for the operations 
$\cdot,{}^{-1},1$, equality $\equiv$, and generators $S$; as well as
certificates \code{asc}, \code{inv}, \code{id}, \code{ref}, \code{sym}, \code{tra}, and \code{cng} of the required
axioms; we refer to \cite{HoTT}*{p.\ 61} for a similar definition in full $\Sigma$,
$\Pi$ notation:

{\small
\begin{align*}
&\code{Group[A]}  :\equiv \left\{\begin{array}{rl|rl}
\cdot & : \code{A}\dashrightarrow\code{A}\dashrightarrow\code{A} 
	& \code{asc} & :\prod_{x,y,z:\code{A}} (x\cdot y)\cdot z \equiv x\cdot (y\cdot z)\\
{}^{-1}& :\code{A}\dashrightarrow\code{A}
	& \code{inv}&: \prod_{x:\code{A}} x^{-1}\cdot x\equiv 1\\
1& : \code{A}
	& \code{id}& : \prod_{x:\code{A}} 1\cdot x\equiv x\\
S&: \code{List[A]}
	& \code{ref} & : \prod_{x:\code{A}}(x\equiv x)\\
\equiv & : \code{A}\dashrightarrow\code{A}\dashrightarrow\code{Boolean}
	& \code{sym} & : {\prod_{x,y:\code{A}}[(x\equiv y)\Rightarrow (y\equiv x)]}\\
&	& \code{tra} & : {\prod_{x,y,z:\code{A}}[(x\equiv y)\times(y\equiv z)\Rightarrow
	x\equiv z]}\\
&	& \code{cng} &: {\prod_{x,y,z,w:A} [(x\equiv y)\times (z\equiv w)\Rightarrow x\cdot z\equiv y\cdot w]}\\
  \end{array}\right.
\end{align*}
}

A type for homomorphisms between terms of \code{Group[A]} and \code{Group[B]} is given by
\[\code{Hom}_{\code{Group}}\code{[A,B]}  :\equiv f :\code{A}\dashrightarrow\code{B} \quad\text{with axiom }\quad  \code{hom} : \prod\nolimits_{x,y:\code{A}} (x\cdot y)f\equiv (x)f\cdot (y)f.\]
 
Intuitively, elements of a group are considered as equivalence classes of words in the generators.
For $G:\code{Group[A]}$ with generators $S$, a {\em straight-line program} is a recursively defined function 
$\sigma:\code{List[A]}\dashrightarrow \code{A}$ using only the functions $\cdot$, ${}^{-1}$, and~$1$; 
cf.\ \cite{Seress}*{p.\ 10}.   If \code{SLP[A]} denotes the type for straight-line
programs, then elements of a group $G$ are interpreted as $\equiv$-equivalence classes of
terms $x:\code{A}$ where $x\equiv S\sigma$ for some straight-line program $\sigma$; that is,
\[
	G= \left\{ x:\code{A} \;\;\middle|\;\; \sum\nolimits_{\sigma:\code{SLP[A]}} x\equiv S\sigma \middle\}\middle/\equiv\right..
\]
Writing $x:\code{A}$ is introducing a term and its type, whereas $x\in G$ asserts a type 
\emph{and} the property (to be assumed or proved) that $x$ is an SLP in the generators of $G$.

Sums and products skip terms of type \code{Nothing}; thus, the axioms need only hold for elements in the group.
Note that \code{SLP[-]} is a type-functor in that, given a 
partial function $f:\code{A}\dashrightarrow \code{B}$, there is an induced map 
$\code{SLP}(f):\code{SLP[A]}\dashrightarrow \code{SLP[B]}$ that replaces the terms of \code{A} with
the corresponding terms in \code{B} assigned by the map $f$; likewise the evaluation-functor 
converts $\code{SLP[A]}\dashrightarrow \code{SLP[B]}$ to $\code{A}\dashrightarrow\code{B}$; these are interchangeable models for \code{Hom}.

Setting $\code{Perm[X]}:\equiv \{ \Delta:\code{List[X]},\;
\cdot:\code{X}\dashrightarrow \code{X}\}$
we can create groups $G:\code{Group[Perm[X]]}$ of {\em permutation type}, 
permuting a list of generic type \code{X}.  The operations for $G$ can be
assigned for all groups of type \code{Perm[X]} rather than individually
for each term, leaving the work of the user to provide generators $S$ and
a congruence $\equiv$.
We can similarly create groups of \emph{matrix type}
$G:\code{Group[Mat[K]]}$; we forgo discussing how to make a type for matrices, fields,
etc.  Groups of \emph{presentation type}    
$G:\code{Group[SLP[Character]]}$  have point-wise 
operations on SLPs; here \code{Character} is a type for a fixed finite alphabet with at least two 
inhabitants.  The group $\mathbb{Z}/d_1\times\cdots\times \mathbb{Z}/d_s$
has type \code{Group[List[Int]]}, and so on.   By saying a group $G$ inhabits a
type with a generic parameter, such as \code{A} above, we mean that we are only
considering functions on $G$ that can be applied to arbitrary substitutions of
\code{A}.  Hence, inhabitants of \code{Group[A]} are
\emph{groups of black-box type} since functions on these types can only apply 
abstract group theory.  

\medskip

{\bf Axioms.}
The inclusion of axioms within our definition of a group type
achieves two aims.  First, it guarantees that our theorems can assume all 
inputs are groups and appropriate functions are homomorphisms, rather than 
simply algebraic objects with  correct signatures.  This is a vital 
difference for our  model that allows us to prove stronger
results than a general black-box group algorithm; again, see Section \ref{sec:Turing}.  Second, 
including axioms is necessary when inputting groups in proof-checkers, which are used to verify complex theory such 
as parts of the Classification of Finite Simple Groups \cite{OddOrder}.
Our model promotes: \emph{Trust, but verify}. 

As a practical matter, for most group types the required information for 
operations and axioms is static and does not need to be provided by the user, 
but is instead part of a computer algebra system such as \textsf{GAP} \cite{gap}.  These terms
are passed along when we create subgroups (replacing generators $S$) and 
quotients (replacing the congruence $\equiv$).  Note that congruences in Type 
Theory would normally be handled with a proposition type, not a Boolean, and thus 
``implication'' can be replaced by functions on propositions.  This is essential
for proof-checkers, but it simplifies our treatment to think of Boolean valued
congruences.   

\medskip

{\bf Timing \& Complexity.}
Inputs are terms of type \code{List[Character]} or 
\code{List[Boolean]}, for example, a text or binary file in a computer.  This 
permits a well-defined notion of input size: the number $\ell$ of terms in the 
list.  An algorithm is a series of functions applied to an input term, and its  timing 
$T(\ell)$ is the number of evaluations as a function of the input size $\ell$. 
An algorithm is in polynomial time if $T(\ell)\in O(\ell^c)$ for a constant
$c$.  In the usual way, we obtain a partial ordering amongst problems, for
example, $A\leq_P B$ says that  whenever problem $B$ can be solved in 
polynomial time, then so can problem $A$. 
Terms depending on generic types, such as 
$G:\code{Group[A]}$, will have varied complexity depending on the properties of 
$\code{A}$; in that case the complexity we prescribe is a function in the number of 
terms of type $\code{A}$ (the {\em arithmetic model}). 

It follows from Pyber's Theorem \cite{BNV:enum}*{p.\ 2} that the number of isomorphism types 
of groups of order $n$ tends to $n^{O(\log^2 n)}$. 
This shows that the minimum possible input length (\emph{Kolmogorov complexity}) 
for a general group of order $n$ with input size $\ell$ is polynomial in $\log n$; 
in particular, timings of the form $O((\ell\log n)^c)$ are always polynomial in the input
size. We often report timings in terms only of the number of factors  $\log n$.

One has to be cautious when estimating running times for function
evaluations.  Take for example $\code{def } f(x:\code{List[A]}):\code{List[A]}:\equiv x+x$, 
the concatenation of a string (list over $\code{A}$) to itself.  It may seem that in $\ell$ recursive 
applications the length of the output is exponentially longer.  In reality, 
either the process of concatenation has to copy every term of $x$ to produce 
the doubling -- which would make the time to recursively evaluate $f$ grow 
exponentially; or, concatenation reuses the value of $x$ (say by two pointers to 
the same string) and thus the final length is $O(\ell)+\code{length}(x)$, not
$2^{\ell}\cdot\code{length}(x)$. A detailed accounting for timing of shared terms is given in 
\cite{BetaReduction}, along with a proof that polynomial time in the Type Theory 
sense agrees with polynomial time in the Turing Machine sense.

\section{Isomorphism testing of black-box abelian groups of most orders}\label{secBBAb}
Our approach for proving  Theorem~\ref{thm:iso-abelian-BB} is to search through 
the primes $p$ less than a fixed bound $c$, and to strip off the $p$-torsion 
subgroups, leaving behind a direct factor containing the Sylow $q$-subgroup of
the group for every prime $q>c$.  Here we rely on the applicable number theory
to observe that what remains is almost always square-free, and consequently 
uniquely characterized by the order;  as a result, isomorphism can be decided. 
The work is to acquire the torsion subgroups for \emph{small} primes $p<c$ 
without involving difficult problems such as integer factorization or
discrete logarithms.  Critical ingredients in this are the progressively 
stronger results on the computability of abelian groups of black-box type, most
recently the work of Karagiorgos \& Poulakis \cite{KP}.

\subsection{Bases and extended discrete logarithms}\label{sec:2wayisom}
Our effort to solve the isomorphism problem for abelian groups relies on a number
of related problems, the first of which are one-way and two-way recognition questions, stated  as follows.

\mydes{1-AbelRecog}
{a finite abelian $G:\code{Group[A]}$;}
{$d_1,\ldots,d_s:\code{Int}$ with $d_1|\cdots |d_s$ and a map 
$\alpha\colon\code{Hom$_{\code{Group}}$[List[Int],A]}$ describing
an isomorphism $\alpha\colon \mathbb{Z}/{d_1}\times\cdots \times\mathbb{Z}/d_s\to G$ 
with polynomial-time evaluation.}

To expose the nuance in {\sf 1-AbelRecog} consider the following stronger
goal.

\mydes{2-AbelRecog}
{a finite abelian $G:\code{Group[A]}$;}
{$d_1,\ldots,d_s:\code{Int}$ with $d_1|\cdots |d_s$, a map 
$\alpha:\code{Hom}_{\code{Group}}\code{[List[Int],A]}$ describing and isomorphism
$\alpha:\mathbb{Z}/{d_1}\times\cdots \times\mathbb{Z}/d_s\to G$ with polynomial-time evaluation, and its inverse
$\alpha^{-1}:\code{Hom}_{\code{Group}}\code{[A,List[Int]]}$ with polynomial-time evaluation.}

One-way isomorphisms are computable only in the direction of
the arrow, inverse images could be hard.  Meanwhile, two-way isomorphisms 
can be used efficiently in both directions.

The classification of finitely-generated abelian groups \cite{handbook}*{Theorem 9.12} yields the following.

\begin{proposition}
Isomorphism testing of finite abelian groups is polynomial-time reducible to problem
{\sf 1-AbelRecog}.
\end{proposition}

We consider how these problems may be solved; we use the following terminology 
for abelian groups $G$. An ordered \emph{basis} $x_1,\dots,x_s\in G$ is defined
by satisfying  $G=\langle x_1\rangle\times \cdots \times \langle x_s\rangle$; 
it is {\em canonical} if each order $|x_i|$ divides $|x_{i+1}|$.  This leads 
to the following computational task:

\mydes{CanonicalBasis}
{a finite abelian $G:\code{Group[A]}$;}
{a canonical basis $x_1,\ldots,x_s:\code{A}$ for $G$ together with the orders
 $|x_1|,\ldots,|x_s|$.}

\begin{proposition}
Problem {\sf 1-AbelRecog} is polynomial-time equivalent to 
{\sf CanonicalBasis}.
\end{proposition}
\begin{proof}
Given a canonical basis $x_1,\ldots,x_s$ of $G$ with known orders, return
$d_i=|x_i|$ together with
$(f_1,\dots,f_s)\mapsto x_1^{f_1}\cdots x_s^{f_s}$. Given a one-way isomorphism
$\alpha: \mathbb{Z}/d_1\times\cdots \times \mathbb{Z}/d_s\to G$ with 
$d_1|\cdots |d_s$, then $g_1,\ldots,g_s$ is a canonical basis for $G$, where each  $g_i=(0,\ldots,1,\ldots,0)\alpha$ with
$1$ in position~$i$.
\end{proof}

The known approaches to solve {\sf CanonicalBasis} in one way or another
reduce to the following problem, which we shall also use.
  
\mydes{ExtendedDiscreteLog (EDL)}
{a basis $x_1,\ldots,x_s:\code{A}$  of a finite abelian group $G:\code{Group[A]}$, and
$g:\code{A}$;}
{$f_1,\ldots,f_s:\code{Int}$ such that $g=x_1^{f_1}\cdots x_s^{f_s}$, or \code{Nothing} if $g\notin G$.}

The next theorem, due to Teske \cite{Teske}*{p.\ 523}, uses baby-step giant-step methods to achieve the stated  bound; let $\epsilon(G,p)$ be a bound on the $\log_p$ of the exponent of $G$, and write $d(G)$ for the size of a minimal generating set of $G$.

\begin{theorem}[(\cite{Teske})]\label{thm:EDL}
For a finite abelian $p$-group  $G:\code{Group[A]}$, there is a deterministic
algorithm to solve {\sf EDL} in time $O(\epsilon(G,p)\lceil p^{1/2}\rceil^{d(G)})$.
\end{theorem}

\begin{remark}\label{remTeske}
It should be emphasized that the difficulty of (extended) discrete logarithm
problems can be confusing, because it is often quantified without explicit 
details about the inputs. For instance, for groups $G:\code{Group[List[Int]]}$
and $G:\code{Group[Perm[A]]}$, Sims  \cite{handbook}*{Section~9.2} has shown a polynomial-time solution based  
on Hermite normal forms.  Likewise, if 
$G:\code{Group[Units[K]]}$ is a group of units of a finite field, or if 
$G:\code{Group[Elliptic[K]]}$ is a group on the points of an elliptic curve (such
as in applications to cryptography), then approaches based on the Number Field
Sieve \cite{Buhler,Gordan} can  
be applied to solve the discrete logarithm problem in expected running time $\exp(\tilde{O}((\log |G|)^{1/3})))$.\footnote{Original work on Number Field Sieves depended on heuristic time
bounds, see \cite{NFS} for rigorous complexity statements.} Teske's  complexity 
$O(\sqrt{|G|})$ for the general case $G:\code{Group[A]}$
indicates that without specific knowledge about parameter type $\code{A}$, the
problem is substantially harder.
\end{remark}

\pagebreak 
\begin{proposition} Problem
{\sf 2-AbelRecog} is polynomial-time equivalent to the pair {\sf EDL} and 
{\sf 1-AbelRecog}.
\end{proposition}
\begin{proof}
If $\alpha$ is a two-way isomorphism, then we also have a one-way isomorphism, and
hence a canonical basis $x_1,\ldots,x_s$. Now $g\in G$  yields 
$g\alpha^{-1}=(f_1,\ldots,f_s)$ and 
$g = g\alpha^{-1}\alpha=x_1^{f_1}\cdots x_s^{f_s}$, so this solves {\sf EDL}.
If we have a one-way recognition 
$\alpha\colon\mathbb{Z}/d_1\times\cdots\times\mathbb{Z}/d_s\to G$, we also have
a canonical basis and hence we can ask to solve {\sf EDL} in $G$; doing so is 
equivalent to asking for inverse images of $\alpha$; this solves {\sf 2-AbelRecog}.   
\end{proof}
 
Applied to abelian groups whose prime divisors are known, Karagiorgos \& Poulakis
\cite{KP} achieve the following; note that in Theorem~\ref{thm:KP} we give a slightly simpler, but less strict bound compared to the one
proved in \cite{KP}.  For an abelian group $G$ denote by $G_p$ its $p$-torsion subgroup; if $\pi$ is a set of primes, then $G$ is a \emph{$\pi$-group} if the prime divisors of the order $|G|$ all lie in $\pi$.  
 
\begin{theorem}[(\cite{KP}*{Theorem~1})]\label{thm:KP}
For a finite set $\pi$ of primes, there is a deterministic algorithm that 
given an abelian black-box type $\pi$-group $G=\langle S\rangle$ of known order
$n$, returns a canonical basis for $G$ in time 
{\small
\begin{equation*}
O\left(|S|(\log n)^3  
+|S|\sum\nolimits_{p\in \pi} {\epsilon(G,p)}\lceil p^{1/2}\rceil^{d(G_p)-1}\right).
\end{equation*}}
\end{theorem} 

As suggested in the timing estimate, the proof of Theorem~\ref{thm:KP} applies {\sf EDL} a polynomial number of times; indeed, that proof  
demonstrates the next proposition.

\begin{proposition}\label{prop:EDL}
For a finite set $\pi$ of primes and an abelian $\pi$-group of black-box type of known order,
problem {\sf CanonicalBasis} (equivalently {\sf 1-AbelRecog}) is polynomial-time 
reducible to {\sf EDL}; in particular, {\sf 2-AbelRecog} is polynomial-time equivalent to 
{\sf EDL}.
\end{proposition}
 
Among the implications of Theorem~\ref{thm:KP} is that {\sf 1-AbelRecog} can be solved 
in polynomial time for cyclic $\pi$-groups of known order $n$, see \cite{KP}*{Corollary~1}.  
Of course, as we have mentioned above, this gives only a one-way isomorphism 
$\alpha: \mathbb{Z}/n\to G$. A two-way isomorphism would be equivalent to solving the
discrete logarithm problem in $G$ with respect to a given cyclic generator.  A further 
nuance is that even small primes present a challenge to computation if they occur in 
large powers.  Perhaps surprising, using {\sf 1-AbelRecog} to verify that an abelian 
group is isomorphic to $C_{2}^n$ seems to require $\Theta(2^{n/2})$ operations.

\subsection{Counting}

We have seen some algorithms that will contribute to 
Theorem~\ref{thm:iso-abelian-BB}; here  we focus on the estimates on the number 
of group orders for which we will apply these algorithms. 

\begin{definition}\label{def:eventually-sq-free}
A positive integer $n$ is \emph{pseudo-square-free} if $n=p_1^{e_1}\ldots p_k^{e_k}b$, 
where $b$ is square-free, every prime divisor of $b$ is bigger than $\log\log n$, the 
primes $p_1,\ldots,p_k\leq \log\log n$ are all distinct and coprime to $b$, and each 
$p_i^{e_i}\leq \log n$.
\end{definition}
 
We often use $b$  for the components of an integer $n$ whose prime divisors 
are {\em b}igger than $\log\log n$; we also use $B$ for a subgroup of a group of size $n$ 
such that $|B|$ has this property.

\begin{lemma}\label{lem:Count-1}
The set $\mathcal{D}$ of pseudo-square-free integers is dense.
\end{lemma}
\begin{proof}
It follows from  Erd\H{o}s-P\'alfy \cite{ErdosPalfy}*{Lemma~3.5}, that almost every integer 
$n$ satisfies the following: if a prime $p>\log\log n$ divides $n$, then $p^2\nmid n$. Thus 
almost every $n$  can be decomposed as $n=p_1^{e_1}\ldots p_k^{e_k}b$ with $b$ square-free 
such that every prime divisor of $b$ is greater than $\log\log n$, and $p_1,\ldots,p_k\leq \log\log n$ 
are distinct primes coprime to $b$. Let $x>0$ be an integer. We now compute an estimate for 
the number $N(x)$ of integers $0<n\leq x$ which are divisible by a prime $p\leq \log \log n$ 
such that the largest $p$-power $p^e$ dividing $n$ satisfies $p^e>\log n$. We want to show 
that $N(x)/x \to 0$ for $x\to \infty$; this proves that for almost all integers $n$, if 
$p^e\mid n$ with $p\leq \log \log n$, then $p^e\leq \log n$. Together with  \cite{ErdosPalfy}*{Lemma~3.5}, 
our claim then follows. 

To get an upper bound for $N(x)$, we consider integers between $\sqrt{x}$ and $x$ with respect to  
the above property, and add $\sqrt{x}$ for all integers between $1$ and $\sqrt{x}$. Note that 
if $p^e\geq \log n$, then $e\geq \log\log n / \log p$. Since we only consider $\sqrt{x}\leq n\leq x$, 
this yields $e\geq c(x)$ where
	\[c(x)= \log\log \sqrt{x}/\log\log\log x.\] 
Note that $c(x)\to \infty$ if $x\to \infty$; now an upper bound for $N(x)$ is
  \begin{eqnarray*}
    N(x) &\leq & \sqrt{x} + \sum\nolimits_{k=2}^{\lfloor \log \log \sqrt{x}\rfloor}\frac{x}{k^{c(x)}}\\
    &\leq & \sqrt{x} + x \int_2^{\log\log\sqrt{x}} \frac{1}{y^{c(x)}}\text{d}y\\
    &=& \sqrt{x} +  x\cdot \frac{1}{1-c(x)}\left[\frac{1}{(\log\log\sqrt{x})^{c(x)-1}}-\frac{1}{2^{c(x)-1}}\right].
  \end{eqnarray*}
Since $1/(1-c(x))\to 0$ from below, we can estimate:
  \begin{eqnarray*}
    N(x) 
    &\leq & \sqrt{x} +  x\cdot \left|\frac{1}{1-c(x)}\right|\left[-\frac{1}{(\log\log\sqrt{x})^{c(x)-1}}+\frac{1}{2^{c(x)-1}}\right]\\
    &\leq &\sqrt{x} +  x\cdot \left|\frac{1}{1-c(x)}\right|\left[\frac{1}{2^{c(x)-1}}\right],
  \end{eqnarray*}
so $N(x)=o(x)$ since
\[N(x)/x \leq \sqrt{x}/x + \left|\frac{1}{1-c(x)}\right|\left[\frac{1}{2^{c(x)-1}}\right]\to 0\quad \text{if $x\to \infty$}.\qedhere\]
  \end{proof}

\subsection{Proof of Theorem~\ref{thm:iso-abelian-BB} (Isomorphism testing of abelian groups)}
 
We need the following preliminary result.
  
\begin{lemma}\label{lemnab}
Fix $c>1$. There is a polynomial-time algorithm that given a natural number~$n$, 
returns all primes $p<c$, and factors $n=ab$ such that the  prime divisors of $a$ 
and $b$ satisfy $p\leq c$ and $p>c$, respectively. 
\end{lemma}
\begin{proof} 
  Let $n_p$ be the largest $p$-power dividing $n$. Let $\pi$ be the set of all primes $p\leq c$; using the Sieve of Eratosthenes, $\pi$  can be  determined in $O(c\log\log c)$ steps. Running over all primes $p\leq c$, we compute $A=\{n_p : \text{$p\leq c$ a prime}\}$, and return $a=\prod_{n_p\in A} n_p$ and $b=n/a$.
\end{proof}

\begin{proof}[Proof of Theorem~\ref{thm:iso-abelian-BB}]
  We prove the theorem for the set $\mathcal{D}$ of pseudo-square-free integers, which we argued in Lemma~\ref{lem:Count-1} is a dense set of integers.
  \smallskip

{\em Algorithm.} Let $G=\langle S\rangle$ and $\tilde{G}=\langle \tilde{S}\rangle$ be abelian black-box type groups of  known order  $n\in\mathcal{D}$. Use Lemma \ref{lemnab} to write $n=ab$ where the prime factors $p$ of $a$, resp. $b$, satisfy $p\leq \log\log n$, or $p>\log\log n$ respectively. Now set $A=\langle s^b : s\in S\rangle$ and $\tilde{A}=\langle s^b : s\in \tilde{S}\rangle$, and use the algorithm of Proposition~\ref{prop:EDL} (via Theorem~\ref{thm:KP}) to decide if there is an isomorphism $A\to \tilde{A}$.  If so return \code{True}, and otherwise \code{False}.\smallskip

{\em Correctness.}  The proof hinges on the assumption that $n$ is pseudo-square-free.  Since $a$ and $b$ are coprime, we can decompose $G=A\times B$ and $\tilde{G}=\tilde{A}\times \tilde{B}$ where $B=\langle s^a:s\in S\rangle$ and $\tilde{B}=\langle s^a:s\in \tilde{S}\rangle$,  so isomorphism is decided by deciding whether $A\cong \tilde{A}$ and $B\cong\tilde{B}$.  Since $|B|=b=|\tilde{B}|$ is square-free, $B\cong \tilde B\cong \mathbb{Z}/b$ holds automatically, so it suffices to test $A\cong \tilde A$.
\smallskip

{\em Timing.} The work of factorization is handled in polynomial time by Lemma~\ref{lemnab}.  
Factoring $a=p_1^{e_1}\ldots p_k^{e_k}$, our assumptions on $a$ imply that for each prime 
divisor $p\mid a$ we have
  \[\epsilon(A,p)\lceil p^{1/2}\rceil^{d(A_p)-1}\leq e_p p^{e_p}\leq (\log\log n)(\log n).\]
Theorem~\ref{thm:KP} now demonstrates that testing $A\cong \tilde A$ is in time $O(|S|(\log n)^2(\log n+ \log\log n))$, hence polynomial in the input size.
\end{proof}  

\begin{corollary}\label{corIF1}
Given an integer factorization oracle, {\sf $1$-AbelRecog} is in Las Vegas polynomial time for groups of pseudo-square-free order.
\end{corollary}
\begin{proof}
  We use the notation of the proof of Theorem \ref{thm:iso-abelian-BB}. Assuming a means to factor integers, we can improve our previous isomorphism test in two ways.  First, we can  verify that $b$ is square-free -- a task which is in general not known 
without factorizing $b$, see \cite{Adleman}.  Second, we can select random elements of 
$B=\langle s^a:s\in S\rangle$ and test their order until we find an element $g$ of order 
$b$, and thus a generator for $B$; (this is the part of the algorithm which makes it Las Vegas). 
Therefore a one-way isomorphism $\mathbb{Z}/b\to B$ is defined by
$\code{def } f(m:\code{Int}):\code{Group[A]}:\equiv g^m$.  Given a one-way isomorphism 
$\alpha:\mathbb{Z}/d_1\times\cdots\times \mathbb{Z}/d_s\to A$, we extend this to
$\gamma:\mathbb{Z}/d_1\times\cdots\times \mathbb{Z}/(d_s b)\to G$ by
\[\code{def }\gamma((x_1,\dots,x_s):\code{List[Int]}):\code{Group[A]}:
	\equiv (x_1,\dots,x_s^b)\alpha\cdot g^{(x_s^{d_s})}.\]
Recall our type $\code{Hom}_{\code{Group}}\code{[List[Int],A]}$ is a pair of a function 
and a certificate that this function is a homomorphism.   To demonstrate how the proof of 
a homomorphism property can be provided as part of the return of the above homomorphism, 
we note that \[(x_1,\ldots,x_s)=\sum\nolimits_{i=1}^s x_i (0,\ldots,\overset{i}{1},\ldots,0)\] 
can be considered as  an SLP in the generators satisfying the relations.  Therefore the 
certificate provides the relations of the group and the assignment of generators, such that the homomorphism property follows from von Dyck's Theorem  \cite{rob}*{2.2.1}.\footnote{Certificates 
of homomorphism in this work are all given by von Dyck's Theorem, so going forward we will 
suppress discussion of this step.  This is, however, a crucial ingredient in making computations 
with checkable proofs.}
\end{proof}

\subsection{Membership and presentations}\label{secMS}
Later we shall need not only isomorphisms, but also presentations and membership tests.
These will be built upon what we can learn about abelian groups.  So we pause to 
inspect the implication of the above isomorphism tests.  

\begin{definition}\label{def:member}
Fix a group $G:\code{Group[A]}$ with generating set $S$.
A {\em constructive membership test} for $G$ is a function
$\Upsilon:\code{A}\dashrightarrow\code{SLP[A]}$ such that for every $x:\code{A}$, 
if  $x\in G$, then $x\equiv (S)\Upsilon_x$, where we abbreviate $\Upsilon_x=(x)\Upsilon$.
\end{definition}

Note that if $\Upsilon_x:\code{SLP[A]}$, instead
of \code{Nothing}, then $(S)\Upsilon_x$ is a word in the generators $S$, hence $(S)\Upsilon_x\in G$ lies indeed in $G$.  Thus, when
$(S)\Upsilon_x\equiv x$ holds, our interpretation of group
elements implies that both $(S)\Upsilon_x$ and $x$ are representatives of the same
equivalence class in $G$, in particular, $x\equiv (S)\Upsilon_x$ implies $x\in G$.

\begin{proposition}\label{prop:abelian-member}
Given a two-way isomorphism from $\mathbb{Z}/d_1\times\cdots \times \mathbb{Z}/d_s$ to 
$G:\code{Group[A]}$, there is a constructive membership test for $G$.
\end{proposition}
\begin{proof} 
Fix a two-way isomorphism 
$\alpha:\code{Hom}_{\code{Group}}\code{[List[Int],A]}$ 
describing an isomorphism from $\mathbb{Z}/{d_1}\times\cdots \times \mathbb{Z}/{d_s}$ to $G$, and let
$\alpha^{-1}:\code{Hom}_{\code{Group}}\code{[A,List[Int]]}$ be its inverse. 
Recall that \code{Nothing} is a permitted output of 
$\alpha^{-1}$ and $\alpha$, see Section~\ref{sec:general-inputs}, so we  define
\[\code{def } \Upsilon(x:\code{A}^?):\code{A}^? :\equiv  
	\text{if } x\equiv ((x\alpha^{-1})\alpha)\text{ then }(x\alpha^{-1})\alpha;\text{ else }\code{Nothing}.\]
Now to see this works let us suppose $x:\code{A}$ and abbreviate $\Upsilon_x=(x)\Upsilon$. 
If $x\equiv (x\alpha^{-1})\alpha=\Upsilon_x$, then $x\in G$ since $\alpha$ is an isomorphism 
onto $G$. Now suppose instead that $x\not\equiv \Upsilon_x$.
This can happen if $(x\alpha^{-1})\alpha$ is \code{Nothing}, or because 
$(x\alpha^{-1})\alpha=x':\code{A}$, but $x\not\equiv x'$.  In the first case observe
that every term $y:\code{List[Int]}$ is in some equivalence class of the group
$\mathbb{Z}/d_1\times\cdots\times \mathbb{Z}/d_s$, and therefore the only input
to $\alpha$ that has an output of \code{Nothing} is \code{Nothing}.  In particular, if
$(x\alpha^{-1})\alpha$ is \code{Nothing}, then $x\alpha^{-1}$ is \code{Nothing}, which
excludes $x$ from the domain of $\alpha^{-1}$;  by assumption, $\alpha^{-1}$ is 
an isomorphism on $G$, and so its domain includes $G$: it follows that $x\notin G$.
In the remaining case, $x'=\Upsilon_x=(x\alpha^{-1})\alpha\not\equiv x$, but $x'\in G$ 
since it is in the image of $\alpha$.   Hence, $(x'\alpha^{-1})\alpha =x \not\equiv (x\alpha^{-1})\alpha=x'$, 
but since $x'\in G$ and $\alpha$ and $\alpha'$ are isomorphisms on $G$, this proves the negation. 
Thus, if $x\not\equiv \Upsilon_x$, then $x\notin G$.

Lastly, if $\Upsilon_x$ is not \code{Nothing}, then we can consider $\Upsilon_x:\code{SLP[A]}$, 
since in this case $\Upsilon_x=(x\alpha^{-1})\alpha$, where  $x\alpha^{-1}=(x_1,\ldots,x_s):\code{List[Int]}$ 
can be interpreted as an SLP describing an element in the group $\mathbb{Z}/{d_1}\times\cdots \times \mathbb{Z}/{d_s}$; 
by construction, this SLP also describes $x$ in $G$. 
\end{proof}

Next, we describe Luks' \emph{constructive presentations} \cite{Luks:mat}.

\begin{definition}[(\cite{Luks:mat}*{Section~4.2})]\label{def:pres}
Let $G$ be a group and $N\lhd G$. A {\em constructive presentation} of a group $G/N$ 
is a free group $F_X$ on a set $X$, a homomorphism $\phi\colon F_X\to G$, a function 
$\psi\colon G\to F_X$, and a set $R\subset F_X$ such that
\begin{items}
\item[(i)] for every $g\in G$, $g^{-1}(g\psi\phi)\in N$;
\item[(ii)] $N\phi^{-1}=\langle R^{F_X}\rangle$, the normal closure of $\langle R\rangle$ in $F_X$.
\end{items}
\end{definition}
To offer some perspective on this definition, when the objects above are translated
into types with which we wish to compute, then the meaning is as follows. First,
$\langle X\mid R\rangle$ is a generator-relator presentation of the group 
$G/N$, see \cite{Luks:mat}*{Lemma 4.1}.  To make $\phi$, 
begin with  $\phi_0:\code{Character}\dashrightarrow\code{A}$ 
that assigns the generators $X$ of $F_X:\code{Group[SLP[Character]]}$ 
to the image $S\subset G:\code{Group[A]}$; now $\phi$ can be implemented by
applying the SLP functor $\Phi=\code{SLP}(\phi_0):\code{Hom}_{\code{Group}}\code{[SLP[Character],A]}$, see  Section~\ref{sec:general-inputs}.  
The interesting part is $\psi$, which can be implemented by $\Psi:\code{A}\dashrightarrow\code{SLP[Character]}$; this is in general not a homomorphism,
but serves to write terms $g:\code{A}$ as the image of an SLP $\sigma\in F_X$ 
such that $g\equiv g\psi\phi\bmod{N}$, or as \code{Nothing} if it is determined that $g\notin G$.   

\begin{proposition}\label{prop:abelian-pres}
Given a two-way isomorphism from $\mathbb{Z}/d_1\times\cdots \times \mathbb{Z}/d_s$ to 
$G:\code{Group[A]}$, there is a constructive presentation for $G$.
\end{proposition} 
\begin{proof}
Fix a two-way isomorphism 
$\alpha:\code{Hom}_{\code{Group}}\code{[List[Int],A]}$ 
describing an isomorphism from $\mathbb{Z}/{d_1}\times\cdots \times \mathbb{Z}/{d_s}$ to $G$, and let
$\alpha^{-1}:\code{Hom}_{\code{Group}}\code{[A,List[Int]]}$ be its inverse. A presentation for $G$ is 
$\langle x_1,\ldots,x_s \mid x_1^{d_1},\ldots,x_s^{d_s},[x_i,x_j]\text{ for all $i,j$}\rangle$. We now make it constructive; note that the normal subgroup is $N=1$. For $i=1,\ldots,s$ define $\phi_0(x_i):\equiv (0,\ldots,1,\ldots,0)\alpha$ with $1$ in position $i$, and then let $\phi=\code{SLP}(\phi_0)$, describing a map $F_X\to G$. As mentioned above,  $\phi$ can be certified  as a homomorphism by von Dyck's theorem  \cite{rob}*{2.2.1}. Next, we define
\[\begin{array}{lcl}
\code{def }\psi(x:\code{A}^?):\code{SLP[Character]}^? &:\equiv & \text{if }x\alpha^{-1}=(k_1,\ldots,k_s):\code{List[Int]}\text{ and }x\equiv x\alpha^{-1}\alpha\\&&\text{then } x_1^{k_1}\cdots x_s^{k_s}:\code{SLP[Character]}\text{ else }\code{Nothing}
\end{array}\]
The logic here is similar to that of the proof of Proposition~\ref{prop:abelian-member}, in that
we use $x\equiv x\alpha^{-1}\alpha$ to determine that inputs lie in $G$, and when they do
we extract an appropriate SLP.
\end{proof}

\section{Isomorphism testing of meta-cyclic  groups}\label{secBBmc}

In proving Theorem~\ref{thm:iso-abelian-BB} we retained deterministic
and polynomial time steps, because we avoided dealing with the cyclic factor of large primes.  
In Theorem~\ref{thm:iso-meta-cyclic-BB}, for meta-cyclic groups, this is no longer possible.
We are forced to either give up on determinism or on efficiency.  
Our approach will be to make a probabilistic algorithm. 

\medskip

\begin{center}
  \begin{tabular}{ll}
{\bf Note:\hspace*{-0.1cm}}&Throughout this section we assume we know the prime factors of $n=|G|$.
  \end{tabular} 
\end{center}

\subsection{Isolating large primes}

For a number $c$ define
\begin{align*}
	\pi(> c) & = \{n : \text{every prime divisor $p$ of $n$ satisfies $p> c$}\}\\
	\pi(\leq c) & =  \{n : \text{every prime divisor $p$ of $n$ satisfies $p\leq c$}\}.
\end{align*}
We prove in Theorem \ref{thm:mostly-split}  that for almost all orders $n$, a solvable 
group of size $n$ has a normal Hall $\pi(>\log\log n)$-subgroup. We start with a few 
preliminary results.

\begin{definition}\label{defKBfree}
  An integer $n$ is {\em $(k,c)$-free} if $p\leq c$ for every prime $p$ with $p^k\mid n$; the integer $n$  is {\em $c$-separable} if for all prime divisors $p>c$ and $q$, we have that $q^k\mid n$ and $q^k\equiv 1\bmod{p}$ imply $k=0$. 
\end{definition}
 
The following result is due to Erd\H os \& Palfy; it follows directly from  \cite{ErdosPalfy}*{Lemmas 3.5 \& 3.6} 

\begin{lemma}\label{lem:Count-2}
The set $\mathcal{D}' = \{n: n\in\mathbb{N}\textnormal{ is $(2,\log \log n)$-free and $\log\log n$-separable}\}$ is dense.
\end{lemma}
 
  Let $\mathcal{D}$ be the set of pseudo-square-free integers as defined in Definition \ref{def:eventually-sq-free}) and let $\mathcal{D}'$ be the set of integers defined in Lemma~\ref{lem:Count-2}.
 
\begin{lemma}\label{lem:Count-3}
The set $\hat{\mathcal{D}} =\mathcal{D}\cap\mathcal{D}'$ is dense.
\end{lemma} 
\begin{proof}
It follows from Lemmas~\ref{lem:Count-1} \& \ref{lem:Count-2} that $\mathcal{D}$ and $\mathcal{D}'$ are dense. Clearly, $\mathcal{D}\cup\mathcal{D}'$ is dense, and an inclusion-exclusion argument proves the claim.
\end{proof}

\begin{lemma}\label{lem:sep-normal}
Let $G$ be a solvable group of $c$-separable order $n$. If $p>c$, then the Sylow $p$-subgroup of $G$ is normal in $G$.
\end{lemma} 
\begin{proof}
  Let $q\ne p$ be a prime dividing $n$, and let $H$ be a Hall $\{p,q\}$-subgroup of $G$ of order $p^eq^f$; see \cite{handbook}*{Section 8.10.1}. The Sylow Theorem \cite{handbook}*{Theorem 2.19} shows that the  number $h_p$ of Sylow $p$-subgroups of $H$ divides $q^f$ (and hence $n$) and is congruent to $1$ modulo $p$.  Since $n$ is $c$-separable, it follows that $h_p=1$, so $H$ has a normal Sylow $p$-subgroup. Now fix a Sylow basis $\mathcal{P}=\{P_1,\dots,P_s\}$ for $G$, that is, a set of
  Sylow subgroups, one for each prime dividing $n$, such that $P_iP_j=P_jP_i$ for all $i$ and $j$; see \cite{rob}*{Section 9.2}. Let $P=P_u$ be the Sylow $p$-subgroup for $G$ in $\mathcal{P}$.
Since $G=P_1\cdots P_s$, every $g\in G$ can be written as $g=g_1\ldots g_s$ with each $g_j\in P_j$.  Since $PP_j=P_j P$, the group $PP_j$ is a 
Hall $\{p,p_j\}$-subgroup of $G$. As shown above, $P$ is normal in $PP_j$; in particular, $g_jP=Pg_j$ for every $j$. Consequently, $gP=g_1\ldots g_s P = Pg_1\ldots g_s=Pg$, which proves that $P$ is normal in $G$.
\end{proof}

\begin{theorem}\label{thm:mostly-split} 
Every solvable group $G$ of order $n\in\hat{\mathcal{D}}$ has a normal Hall $\pi(> \log\log n)$-subgroup $B$ and a complementary Hall $\pi(\leq \log\log n)$-subgroup~$K$.
\end{theorem}
\begin{proof}
Let $c=\log\log n$ and note that $n$ is $c$-separable by assumption. By Lemma~\ref{lem:sep-normal}, 
for every prime $p>c$ there is a normal Sylow $p$-subgroup $P_p$ of $G$; thus the 
product $B$ of all these normal Sylow subgroups is a normal Hall $\pi(>c)$-subgroup 
of $G$. The Schur-Zassenhaus Theorem \cite{rob}*{(9.1.2)} shows that there is a 
complementary Hall $\pi(>c)'$-subgroup $K$ to $B$. 
It follows from Lemma  \ref{lem:Count-3} that $\hat{\mathcal{D}}$ is dense.
\end{proof}

\subsection{Recognizing coprime meta-cyclic decompositions}
The following theorem is proved at the end of this section; we assume the notation 
of the theorem throughout this section.  We say a group $G$ is \emph{coprime meta-cyclic}
if $G=U\ltimes K$ for cyclic $U,K\leq G$ of coprime order. Recall the definition of $\hat{\mathcal{D}}$, 
see Lemma~\ref{lem:Count-3}.
 
\begin{theorem}\label{thm:meta-cyclic-recog}
	There is a Las Vegas polynomial-time algorithm that
	given a solvable group $G:\code{Group[A]}$ with known factored order $n\in \hat{\mathcal{D}}$,
	decides if $G$ is coprime meta-cyclic, and, if so, returns $U,K\leq G$ 
	such that $G=U\ltimes K$, where $K$ is the product of all normal Sylow subgroups of $G$, 
	a two-way isomorphism $\alpha\colon \mathbb{Z}/c\to U$, and a 
	one-way isomorphism $\beta\colon \mathbb{Z}/d\to K$.
\end{theorem}
To place this in context, consider the following diagram of short exact sequences:
\begin{align}\label{eq:split}
\xymatrix{ 
1 \ar[r] & {\mathbb{Z}/be} \ar[d]^{\beta}\ar[r] & \mathbb{Z}/c\ltimes_{\theta}{\mathbb{Z}/be}
	\ar[d]^{\gamma}\ar[r]
	& \mathbb{Z}/c \ar[d]^{\alpha}\ar[r] & 1\\
1 \ar[r] & B\times{E} \ar[r] & G\ar[r] & G/(B\times {E}) \ar[r] & 1
}
\end{align}
Theorem~\ref{thm:meta-cyclic-recog} constructs the homomorphisms $\alpha$ and $\beta$, but it does not describe $\theta$ nor $\gamma$
which are constructed later in Theorem~\ref{thm:deconj} and the proof of
Theorem~\ref{thm:iso-meta-cyclic-BB}.  In the group $K=B\times E\cong \mathbb{Z}/be$ 
given in the diagram, $B\cong \mathbb{Z}/b$ is the 
unique Hall $\pi(> \log \log n)$-subgroup and $E\cong \mathbb{Z}/e$ is a cyclic subgroup
of order coprime to $bc$. Note $d=be$ and $B$ must be cyclic since $n\in\hat{\mathcal{D}}$.
That $E$ and $G/(B\times E)$ are cyclic is determined by our algorithm.

Our process is in two steps.  We first recognize $B$
(see Lemma~\ref{lem:big-member}) and $G/B$ (see Lemma~\ref{lem:small-top}) by 
constructing presentations based  
on our abelian isomorphism test (Theorems~\ref{thm:iso-abelian-BB} \& \ref{thm:KP}).
The second stage is to determine whether {$G/B\cong \mathbb{Z}/c\ltimes \mathbb{Z}/e$
with $\gcd(c,e)=1$}, and use this to decide if $G\cong \mathbb{Z}/c\ltimes \mathbb{Z}/be$.
Having removed the presence of large primes in $|G:B|$, this second step can use Sylow 
subgroups to construct a canonical such decomposition (see Lemma~\ref{lem:max-split}).

In what follows we focus on groups of black-box type. We devise an algorithm that 
works through successive quotients of a group $G:\code{Group[A]}$.  While it would be 
standard in group theory to express this with a normal subgroup $N$, where  $G/N$ is 
the quotient, for computations we must expressly convert $G/N$ into a congruence such 
that $x\equiv y$ in $G/N$ if and only if  $x^{-1}y\in N$; cf.\ the definition 
of \code{Group[A]} in Section~\ref{sec:general-inputs}. For that it suffices to use a 
membership test, not to be confused with the stronger constructive membership tests of
Definition~\ref{def:member}.  

\begin{definition} 
Given a group $G:\code{Group[A]}$, a subgroup $N$ is \emph{recognized} if
there is a partial function $\Upsilon:\code{A}\dashrightarrow\code{Boolean}$ such that if 
$g\in G$, then  $g\in N$ if and only if $\Upsilon_g=1$ ($\code{True}$); again, 
we abbreviate $\Upsilon_g = (g)\Upsilon$.
\end{definition}

In this situation, $N$ is not necessarily a black-box group  as we need not to know generators.  

\begin{lemma}\label{lem:big-member} 
	Fix $G:\code{Group[A]}$ and $d:\code{Int}$. 
	If $G$ has a unique Hall $d$-subgroup $H$, then $H$ is recognizable.
	In particular, if $G$ is solvable of order dividing some $n\in \hat{\mathcal{D}}$,
	then $G$ has a  unique Hall $\pi(>\log\log n)$-subgroup and this group is recognizable.
\end{lemma}
\begin{proof}
	To describe the recognition test, suppose $g:\code{A}$.  Test if $g^d=1$ and return
	the result.  This works because we only need to guarantee the result for $g\in G$
	(that is we make no promise about the outcome when $g$ lies outside $G$;
	if $g\notin G$, then we can return anything including \code{Nothing}).
	In this case,  $g\in H$ implies $g^d=1$ since $|H|=d$.  Conversely, if $g^d=1$, 
	then $g$ lies in  some Hall $d$-subgroup of $G$, which by hypothesis is unique.

	For the case when $G$ is solvable of order dividing $n\in\hat{D}$, 
	we know that $G$ has a unique Hall $\pi(>\log\log n)$-subgroup $H$, see Theorem \ref{thm:mostly-split}; using Lemma~\ref{lemnab} we 
	can factor $n=ab$ and so determine the order $|H|=b$; thus we can apply the above algorithm.
\end{proof}

The ambiguity of knowing whether $g:\code{A}$ satisfies $g\in G$ when all we have are 
generators of $G$ could cause concern for users and distress for algorithm designers.  
However, we have already shown 
in Proposition~\ref{prop:abelian-member} that this is decidable for abelian groups;  we build 
on that technique. Note that we assume the prime factors of $|G|$ are known
and we will apply Proposition~\ref{prop:abelian-member} in situations where we can prove 
the complexity is polynomial time.

\begin{lemma}\label{lem:small-top}
	Let $n\in \hat{\mathcal{D}}$ with known factorization.  
	There is a randomised polynomial time (in $\log n$) algorithm that 
	given a $\pi(\leq \log\log n)$-group $G:\code{Group[A]}$ of order dividing $n$, 
	certifies that $G$ is solvable, returning a constructive presentation of $G$, 
	the derived series $G=N_1>\cdots >N_{\ell+1}=1$  with two-way 
	isomorphisms $\phi_i\colon \mathbb{Z}/d_{i1}\times\cdots \times \mathbb{Z}/d_{if_i}\to N_i/N_{i+1}$,
	for each $i$, and a constructive membership test for $G$. If $G$ is solvable, then this algorithm is Las Vegas. If $G$ is non-solvable, then this is a Monte Carlo algorithm which can detect that.
\end{lemma}
General constructive membership in groups of black-box type has an extensive history
with several substantive results, cf. \cite{BBS} and bibliography contained therein.
Our self-contained proof below is a simplified take on the Beals-Babai \emph{blind-descent}
strategy.
\begin{proof}
	Let $G=\langle S\rangle$ be a $\pi(\leq \log\log n)$-group of order dividing $n$.  
	We use the Monte Carlo algorithm of 
	\cite{Seress}*{Theorem~2.3.12} to compute generators for $H\leq [G,G]$ 
	such that with high probability we have $H=[G,G]$.  Now make a recursive 
	call on $H$ o prove $H$ is solvable and build a constructive presentation 
	for $H$ together with a constructive membership test $\Upsilon'$ on $H$.  
	If this step aborts, then report that $G$ is probably non-solvable.
	Otherwise, use $\Upsilon'$ to form the group $G/H$.  Now test if $G/H$ is abelian 
	(thus proving $H=[G,G]$), and if not, restart at most $\varepsilon^{O(\log n)}$-time
	at which point the algorithm reports that $G$ is probably non-solvable.
	
	Now when $G/H$ is abelian, use the algorithms 
	of Section \ref{sec:2wayisom} to produce a two-way isomorphism 
	$\alpha:\code{Hom}_{\code{Group}}\code{[List[Int],A]}$ and its inverse
	$\alpha^{-1}:\code{Hom}_{\code{Group}}\code{[A,List[Int]]}$,
	describing $\mathbb{Z}/{d_1}\times\cdots \times \mathbb{Z}/{d_s}\cong G$. 
	Now we can apply Propositions~\ref{prop:abelian-pres} \&
	\ref{prop:abelian-member} to form a constructive presentation for $G/H$ 
	and a constructive membership test $\Upsilon''$.  Luks' constructive 
	presentation extension lemma \cite{Luks:mat}*{Lemma~4.3} can now be used 
	to define a constructive presentation for $G$ from the known constructive 
	presentations for $G/H$ and $H$.  Finally, we define $\Upsilon$ by
	\[
	\begin{array}{lcl}
	\code{def }\Upsilon(x:\code{A}^?):\code{SLP[Character]}^?&:\equiv
		&\text{if }x\equiv \Upsilon''_x(S)\text{ and } 
		x^{-1}\cdot(S)\Upsilon''_x\equiv (S)\Upsilon'_{x^{-1}\cdot (S)\Upsilon''_x}\\
		&&\text{then} \Upsilon''_x\cdot 
			\left(\Upsilon'_{x^{-1}\cdot (S)\Upsilon''_x}\right)^{-1}\text{ else }\code{Nothing}.
	\end{array}
	\]
	We claim that $\Upsilon$ is a constructive membership test for $G$. To prove 
	this,  consider $x:\code{A}$. If $x$ describes an element in $G$, then also 
	$x\in G/H$, and so $x\equiv (S)\Upsilon''_x\bmod{H}$ by definition of 
	$\Upsilon''$.  This implies $y=x^{-1}\cdot (S)\Upsilon''_x\in H$, and so 
	$y\equiv (S)\Upsilon'_y$ by definition of $\Upsilon'$.  Together, we have
	\begin{align*}
		x & \equiv (S)\Upsilon''_x y^{-1}
			=(S)\Upsilon''_x\cdot \left((S)\Upsilon'_{x^{-1}\cdot (S)\Upsilon''_x}\right)^{-1}
			=(S)\left(\Upsilon''_x\cdot \left(\Upsilon'_{x^{-1}\cdot (S)\Upsilon''_x}\right)^{-1}\right),
	\end{align*} 
	so we return the SLP $\Upsilon''_x\cdot (\Upsilon'_{y})^{-1}$. We now  comment 
	on the complexity. The derived series of $G$ (in fact, any properly descending 
	subgroup chain) has length at most $l = O(\log n)$. To guarantee a correct 
	return with probability $1-\varepsilon$,  we need to run the above Monte Carlo 
	algorithm with prescribed probability $\delta=\varepsilon^{l}$.  In the abelian 
	case we can apply the algorithm of Theorem~\ref{thm:KP} in polynomial time, 
	because if $p^e\mid |G|$, then  $p\leq \log \log n$ and $p^e\leq \log n$.
\end{proof}

We record a few consequences of obtaining constructive presentations.  Versions of
these results were provided in \cite{Luks:mat} for the context of groups of matrix
type.  Here we provide versions for groups of black-box type.

\begin{lemma}\label{lem:get-normal}
Let $G:\code{Group[A]}$  with recognizable cyclic normal subgroup $B$. If  we have a constructive presentation of $G/B$, then we can compute generators for $B$ in polynomial time.
\end{lemma}
\begin{proof}
We use the notation of Definition \ref{def:pres} for the constructive presentation of $G/B$. 
For each $x\in X$ we can find $g_x\in G$ such that $x\phi \equiv g_x\bmod B$, 
namely $g_x= x\phi\psi\phi$.  Use these elements to define
$S=\{g_x : x\in X\}$. For a word $w$ in $X$, the evaluation $w(S)$ denotes replacing each 
$x\in X$ with $g_x\in S$.  Note  that $B$ is the normal closure of $\{w(S): w\in R\}$ in the 
group $G$. By assumption, $B$ is cyclic, and therefore there is a unique
subgroup of each order dividing $|B|$;  in particular, for each $w\in R$ and  $g\in G$, we have
$\langle w(S)\rangle=\langle w(S)^g\rangle$. Thus we can generate $B$ as  $B=\langle w(S):w\in R\rangle$. 
\end{proof}

\begin{lemma}\label{lem:sylow-small}
   Let $G:\code{Group[SLP[A]]}$ be a solvable group of order dividing $n$, 
   where the prime factors of $n$ are known. There is a polynomial-time algorithm that
	given a prime $p$, a constructive 
	presentation for $G$, the derived series $G=N_1>\cdots >N_{\ell+1}=1$, and  two-way 
	isomorphisms $\phi_i\colon \mathbb{Z}/d_{i1}\times\cdots \times \mathbb{Z}/d_{if_i}\to N_i/N_{i+1}$ for each $i$,
	decides if $G$ has a unique Sylow $p$-subgroup $P$, and, if so, returns generators for $P$.
\end{lemma}
\begin{proof}
	The method is a simplification of the Frattini method of Kantor, e.g. \cite{KLM},
	specializing to solvable groups and the case of normal Sylow subgroups.

        \smallskip
        
    \emph{Algorithm.}
    Make a recursive call to $N_2$ to decide if $N_2$ has a unique
    Sylow $p$-subgroup $Q$, and if not then report that $G$ does not
    have a unique Sylow $p$-subgroup.  Otherwise, compute the largest 
    divisor $k$ of $n$ coprime to $p$. Within 
    $\mathbb{Z}/d_{11}\times\cdots \times \mathbb{Z}/d_{1f_i}$
	compute a generating set $S$ for its unique Sylow $p$-subgroup.
	For each $x\in S$, let $g_x\in G$ be a representative of the 
	coset $(x\phi_1)\in G/N_2$.  Set $h_x = g_x^{k}$.
	Set $P=\langle Q,h_x : x\in S\rangle$.  Use the algorithm 
	of Lemma~\ref{lem:small-top} to test if $P$ has order a power of $p$
	and use the membership test on $P$ to prove that  each generator $y\in G$ satisfies $P^y=P$.
	If so, return $P$; otherwise report that $G$ does not have a unique
	Sylow $p$-subgroup.

\smallskip
         
\emph{Correctness.}
	Note that our algorithm returns a normal $p$-subgroup $P$ of $G$.  We claim
	it is in fact a Sylow $p$-subgroup.  First observe that $Q\leq P$. Since $Q\unlhd N_2$ is the unique Sylow $p$-subgroup of $N_2$, it follows that $P\cap N_2\leq Q$. In conclusion, $Q=P\cap N_2$, and so  $|N_2:P\cap N_2|=|PN_2:P|$ is
	prime to $p$.  Moreover, for each $x\in S$ we have $\langle h_x N_2\rangle =\langle x\phi_1\rangle$,  hence 
	$PN_2/N_2=\langle S\rangle \phi_1$ is the Sylow $p$-subgroup of $G/N_2$, so $|G:PN_2|$ is prime to $p$.  In conclusion, $|G:P| = |G:PN_2| |PN_2:P|$ is
	prime to $p$, hence it follows that $P$ is a Sylow
	$p$-subgroup of $G$. Since $P\unlhd G$, the Sylow $p$-subgroup of $G$ is unique.

\smallskip
        
\emph{Timing.}
	The algorithm uses polynomial time routines and makes a most $\log |G|$ 
	recursive calls.  So in total it is a polynomial-time algorithm.
\end{proof}

\begin{lemma}\label{lem:max-split}
	If $G\cong \mathbb{Z}/a\ltimes \mathbb{Z}/b$ with $\gcd(a,b)=1$, then 
	$G\cong \mathbb{Z}/c\ltimes \mathbb{Z}/d$ with $\gcd(c,d)=1$ where 
	where $\mathbb{Z}/d$ is the product of all normal Sylow subgroups of $G$;
	in particular, $b\mid d$, and $c$ and $d$ are determined by the isomorphism type of $G$.
\end{lemma}
\begin{proof}
	By assumption, the Sylow $p$-subgroups of $G$ are isomorphic to ones in either 
	$\mathbb{Z}/a$ or $\mathbb{Z}/b$; as both these groups are cyclic, all 
	Sylow $p$-subgroups of $G$ are cyclic.  Initialize the subgroup $Z=1$ of $G$.  
	For each prime divisor $p$ of $|G|$ compute a Sylow $p$-subgroup $P$ of $G$; 
	if $P$ is normal in $G$, replace $Z$ by $ZP$.  We claim that at each iteration,
	$Z$ is normal, cyclic, and a Hall subgroup of $G$: this is true for $Z=1$. 
	Now in an iteration step, by the induction hypothesis, $|Z|$ is coprime to $|G:Z|$, 
	hence also 	to $|P|$, which shows that $Z\cap P=1$.  Both $Z$ and $P$ are normal 
	and cyclic, so $[Z,P]=1$, and therefore $ZP=Z\times P$ is normal and cyclic; in particular,  $ZP$ is 
	a Hall subgroup. At the end of the iteration, $Z$ contains all normal Sylow 
	subgroups of $G$.  Since $Z$ is a normal Hall subgroup, the  Schur-Zassenhaus 
	Theorem  \cite{rob}*{(9.1.2)} shows that $Z$ has a complement $U$ in $G$, 
	thus $G=UZ=U\ltimes Z$. Since $\mathbb{Z}/b\leq Z$ and 
	$G/(\mathbb{Z}/b)\cong \mathbb{Z}/a$ is cyclic, it follows that $U$ is cyclic. 
	Thus $G\cong \mathbb{Z}/c\ltimes \mathbb{Z}/d$ with $c=|U|$ and $d=|Z|$ coprime.
\end{proof} 

We can now prove the main result of this section.

\begin{proof}[Proof of Theorem~\ref{thm:meta-cyclic-recog}] 
	First use Lemma~\ref{lemnab} to factor $n=ab$ where $b\in \pi(>\log\log n)$
	and $a\in\pi(\leq \log\log n)$.   Use the algorithms of Lemmas~\ref{lem:big-member} 
	and \ref{lem:small-top} to constructively recognize $G/B$ where $B$ is the Hall 
	$\pi(>\log\log n)$-subgroup of $G$ of order $b$. Now we use Lemma \ref{lem:get-normal} to get a generating set for $B$. This allows us to compute (pseudo-)random elements of $B$ until we find a cyclic generator: recall that we assume knowledge of the prime divisors of $n$, so for a random $g\in B$ we can check whether $g^{b/p}\ne 1$ for all prime divisors of $b$; if this holds, then $B=\langle g\rangle$ is determined. Note that the number of generators of 
$(\mathbb{Z}/b)^\times$ is $\varphi(b)/b\in \Omega(1/\log\log b)$, see \cite{euler}*{Theorem~15}, so finding $g$ can be done in Las Vegas  polynomial time. This then allows us to construct a one-way isomorphism 
	$\beta_b\colon\mathbb{Z}/b\to B$. It follows from Theorem~\ref{thm:mostly-split} and the assumptions on $n\in \hat{\mathcal{D}}$ that such a $B$ and isomorphism exist.     Note that if $G$ is coprime meta-cyclic, then every Sylow subgroup of $G$ is cyclic, 
    see the proof of Lemma~\ref{lem:max-split}. Using Lemma~\ref{lem:sylow-small}, for 
    each prime divisor $p$ of $a$, decide if the Sylow $p$-subgroup $P$ of $G/B$ is normal, 
    and if so, use the constructive presentation to decide if $P$ is abelian. In that case, use 
    {\sf 2-AbelRecog} to verify that $P$ is cyclic and to construct a two-way
    isomorphism $\beta_p\colon\mathbb{Z}_{p^{e(p)}}\to P$; this can be done in polynomial time since $p\leq \log \log n$. If we find a non-cyclic Sylow subgroup, then $G$ is not coprime meta-cyclic.

	Let $S$ be the 	set of those primes $p$ for which $\beta_p$ has been constructed; 
	for each $p\in S$, let $g_p\in G$ be a preimage to $1\beta_p$ (using the 
	constructive presentation for $G/B$) and set $x_p={g_p}^{n/p^{e(p)}}$. 
	Note that $G$ splits over $B$, thus the group  $K=\langle B,x_p : p\in S\rangle$ 
	is a Hall subgroup of $G$. Furthermore, composing $1\mapsto 1\beta_p\mapsto x_p$ 
	gives a one-way isomorphism  
	$\hat{\beta}_p\colon \mathbb{Z}/p^{e(p)}\to \langle x_p\rangle\leq K$.
	Hence, $\beta=\beta_b\prod_{p\in S} \hat{\beta}_p$ is a one-way isomorphism
	$\mathbb{Z}/d\to K$ where $d=b\prod_{p\in S} p^{e(p)}$. 
	Along with generators for $K$, we also know the order, and thus $K$ is 
	recognizable by Lemma~\ref{lem:big-member}.

	Lastly, note that if $G$ is coprime meta-cyclic, then Lemma~\ref{lem:max-split} 
	shows that $G=U\ltimes K$ where  $U$ is a cyclic complement to the Hall subgroup $K$.	
	We now attempt to construct this complement.  Use Lemmas~\ref{lem:big-member} and 
	\ref{lem:small-top}  to constructively recognize $G/K$ together with a two-way 
	isomorphism $\phi_1\colon\mathbb{Z}/{d_1}\times\cdots \times \mathbb{Z}/{d_s}\to G/K$, 
	where each $d_i$ divides $d_{i+1}$.  If $s>1$, then $G/K$ is not cyclic and we  report 
	that $G$ is not coprime meta-cyclic. If $s=1$, then set $g\in G$ to be a preimage of 
	$1\phi_1$, and define $U=\langle x\rangle$ where $x=g^{d}$; return $G=U\ltimes K$. 
	Evidently, $|x|=d_1$ because this is the order of $G/K\cong \mathbb{Z}/{d_1}$.
	We further see an evident one-way isomorphism $\mathbb{Z}/{d_1}\to U$, and because
	$d_1\in \pi(\leq \log\log n)$, we can apply {\sf 2-AbelRecog} in polynomial 
	time to make this a two-way isomorphism.
\end{proof}

\subsection{Deconjugation}
Having now the ability to decide if a group is coprime meta-cyclic, and, if so, to 
construct such an extension, it remains to decipher the action of one cyclic group
on another.  We do this with a process we call \emph{deconjugation}.

\begin{theorem}\label{thm:deconj}
Let $G:\code{Group[A]}$ with generating set $\{x,y\}$, such that $|x|=a$ and $|y|=b$ 
are coprime, and $x^{-1} y x\in \langle y\rangle$. Assuming we know the prime factors
of $a$ and $b$, Algorithm \ref{algo:recog-ZfZm}  is Las Vegas and returns an integer $v$ such 
that  $y^x=y^v$ in $O((\log a)(\log b)\nu(\log \nu)^2(\log a)^2+(\log a)^2(\log b)^2)$ 
group operations, where $\nu$ is the largest prime divisor of $a$. 
\end{theorem} 

\begin{algorithm}[!htbp]
  {\small
  \caption{{\sf Deconjugate}}
  \label{algo:recog-ZfZm}
  \begin{algorithmic}[1]
  \Require $G:\code{Group[A]}$ with generators $x,y$ such that $|x|=a$ and $|y|=b$ are coprime, and $y^x\in\langle y\rangle$
  \Ensure $v:\code{Int}$ such that $y^x=y^v$

  \vspace*{-2ex}

\begin{tabbing}
\hspace*{1em}\=\hspace*{1em}\=\hspace*{1em}\=\hspace*{1em}\=\hspace*{1em} \\
\code{\bf def } $\code{Deconjugate}(\langle x,y\rangle:\code{Group[A]} ):\code{Int} :\equiv$ \\
\> {\bf if} $a=p^e$ and $b=q^f$, where $p$ and $q$ are primes, {\bf then}\\ 
\>\> find $k\in\mathbb{Z}/b$ and $g\geq 1$ such that $k^{(p^g)}\equiv 1\bmod{b}$ and $k^{(p^{g-1})}\not\equiv 1\bmod b$\\
\>\> set $h=\min\{e,g\}$, and if $g>e$, then replace $k$ by $k^{(p^{g-e})}$\\
\>\> for all $i\in \{0,1,\ldots,h-1\}$ find $m(i)\in\{0,\dots,p-1\}$ with
$y^{(k^{(p^{h-i}\sum_{j=0}^{i-1}  p^{j}m(j))})}=x^{-a/p^i} y x^{a/p^i}$\\
\>\> {\bf return} $k^{\sum_{i=0}^{h-1} m(i)p^{i}}\bmod b$\\[2pt]
\> {\bf if} $a=uv$ where $1=\gcd(u,v)=us+vt$, {\bf then} \\
\>\> $m(u)=\code{Deconjugate}(\langle x^u,y\rangle)$\\
\>\> $m(v) = \code{Deconjugate}(\langle x^v,y\rangle)$\\
\>\> {\bf return} $m(u)^{s}m(v)^{t}\bmod b$\\
\> {\bf if}  $a=p^e$ and $b=uv$ with $1=\gcd(u,v)=us+vt$, {\bf then}\\
\>\> $m(u)=\code{Deconjugate}(\langle x,y^{u}\rangle)$\\[2pt]
\>\> $m(v)= \code{Deconjugate}(\langle x,y^{v}\rangle)$\\[2pt]
\>\> {\bf return} $m(u)s+ m(v)t  \bmod b$
\end{tabbing}
  \end{algorithmic}}
\end{algorithm}
\begin{proof}
Algorithm \ref{algo:recog-ZfZm} describes \code{Deconjugate}; we now prove that this 
algorithm is correct. An example application is given in Example \ref{exDec} below. 
Let us consider first the base case $a=p^e$ and $b=q^f$ with both $p$ and $q$ 
prime.  Assuming $q=2$, it follows that $\Aut(\langle y\rangle)\cong
C_2\times C_{2^{f-2}}$ and so $x$ must centralize $y$.  In this case 
the algorithm computes $m(i)=0$ for each $i$ and returns $0$.  Now suppose $q>2$. In 
this case, the automorphism group $\Aut(\langle y\rangle)\cong  
(\mathbb{Z}/{q^f})^{\times}\cong C_{q-1}\times C_{q^{f-1}}$ is cyclic.
In particular, there is a unique, cyclic, Sylow $p$-subgroup $\mathbb{Z}/p^g$ of 
$\Aut(\langle y\rangle)$, and the algorithm begins by locating a generator $k$ for the 
unique subgroup isomorphic to  $\mathbb{Z}/p^h$ where $h=\min \{g,e\}$.  Thus, 
$y^x=y^{(k^m)}$ for a unique $m\in\mathbb{Z}/p^h$ with $h\leq e$;  in the 
following write $m=\sum_{i=0}^{h-1} m(i)p^i$ with 
each $m(i)\in \{0,\ldots,p-1\}$.  It follows that for each $i$ we have
\begin{align*}
	x^{(-p^{h-i})} y x^{(p^{h-i})} & = y^{(k^{(mp^{h-i})})}  = y^{(k^{(p^{h-i}\sum_{j=0}^{i-1} m(j)p^{j})})}.
\end{align*}
The algorithm solves for each $m(i)$ inductively, beginning with $m(0)$.

Next consider the case when $a=uv$ with $u,v>1$, such that $1=\gcd(u,v)=us+tv$ for 
integers $s$ and $t$ determined by the extended Euclidean Algorithm. By induction, we know 
\begin{align*}
	x^{-u} yx^u & = y^{m(u)}, & 
	x^{-v} y x^{v} & = y^{m(v)},\\
	x^{u} yx^{-u} & = y^{m(u)^{-1}}, & 
	x^{v} y x^{-v} & = y^{m(v)^{-1}}.
\end{align*}
We now return $m(u)^s m(v)^t$, which is correct since
\begin{align*}
	x^{-1}yx^1 & = x^{-(us+vt)}yx^{us+vt}= 
		\overbrace{x^{\mp u}\cdots x^{\mp u}}^{|s|}
		\overbrace{x^{\mp v}\cdots x^{\mp v}}^{|t|}
		y
		\overbrace{x^{\pm v}\cdots x^{\pm v}}^{|t|}
		\overbrace{x^{\pm u}\cdots x^{\pm u}}^{|s|}\\
		& = 
		\overbrace{x^{\mp u}\cdots x^{\mp u}}^{|s|}
		y^{m(v)^t}
		\overbrace{x^{\pm u}\cdots x^{\pm u}}^{|s|}= 
		y^{m(u)^s m(v)^t}.
\end{align*}
Last, we assume $a=p^e$ and $b=uv$ with $u,v>1$ and $1=\gcd(u,v)=us+vt$.
By induction, we know $x^{-1} y^u x  = y^{m(u)}$ and $x^{-1} y^v x  = y^{m(v)}$, so we have
\begin{align*}
	x^{-1} y x & = x^{-1} y^{su+vt} x
		= (x^{-1} y^{us} x)(x^{-1} y^{vt} x)
		= y^{m(u)s+m(v)t}.
\end{align*}
We comment on the timing. The selection of $k$ in each base case $a=p^e$ and $b=q^f$ can 
be done in non-deterministic Las Vegas polynomial time by selecting random elements and testing 
the order: using the known prime factors, we can determine the largest $p$-power $p^g$ dividing 
the order $\phi(b)$ of $(\mathbb{Z}/b)^\times$. Note that the number of generators of 
$(\mathbb{Z}/b)^\times$ is $\varphi(b)/b\in \Omega(1/\log\log b)$, see \cite{euler}*{Theorem~15}, 
so with $O(\log\log b)$ random choices we succeed: if $u\in \mathbb{Z}/b$ with $\gcd(u,b)=1$ is 
random, we test whether $k=u^{\phi(b)/p^g}$ has order $p^g$. If so, then $k$ is a generator of 
the Sylow $p$-subgroup of $(\mathbb{Z}/b)^\times$, and we can replace $k$ by an element of order 
$p^h$ where $h=\min\{g,e\}$.


The work that remains in the base case is to apply $h\leq e$ rounds of searching for the $m(i)$.
Each round produces an integer $v(i)=p^{h-i}\sum_{j=0}^{i-1} m(j)p^{j}$ where the $m(j)$ are 
known for $j<i-1$, and we solve for $m(i-1)\in \{0,\ldots,p-1\}$.  Using repeated squaring, we 
take $O(e\log p)$ operations to create all the $v(i)$. We need at most $p$ attempts to solve for
$m(i-1)$, and each time compute a power with exponent  $k^{v(i)}\bmod{q^f}$ using 
$O(\log v(i))\subset O(e\log p)$ operations. So solving for all the $m(j)$ in a particular base 
case takes at most $O(e^2 p(\log p)^2)\subset O(\nu(\log\nu)^2(\log a)^2)$ operations. 

Finally,  the base cases are pairs $(p^e,q^f)$ where $p^e\mid a$ and $q^f\mid b$ are maximal 
prime powers dividing $a$ and $b$, respectively. So the total number of recursive calls is the product
of the number of distinct prime divisors of $a$ with the number of distinct prime
divisors of $b$, which is in $O((\log a) (\log b))$. In addition, there are $O((\log a)(\log b))$ 
calls to the extended Euclidean Algorithm, which requires at most $O((\log a)(\log b))$ operations each. 
Together, the overall complexity is $O((\log a)(\log b)\nu(\log \nu)^2(\log a)^2+(\log a)^2(\log b)^2)$, as claimed.
\end{proof}

\begin{example}\label{exDec}
Let $\nu=10$ be the bound on the prime divisors of $a=|x|$, and call primes $p\leq \nu$ small. 
Within $\GL_3(541)$ consider 
\[x  = \begin{bmatrix}
	11 & 0 & 0 \\
	0 & 0 & 311 \\
	0 & 311 & 0
	\end{bmatrix}
	\quad\text{and}\quad
	y  = \begin{bmatrix}
	1 & 47 & 494 \\ 0 & 1 & 0 \\ 0 & 0 & 1
	\end{bmatrix}.\]
        We have $|x|=a=2^2\cdot 3^3=108$ and $|y|=b=541$, and factor the small primes  in $541-1=2^2\cdot 3^3\cdot 5$. 
 Taking a random integer modulo $541$, 
say $97$, we calculate $97^{540/2^2}\equiv 489 \bmod{541}$ and
$97^{540/3^3}\equiv 510\bmod{541}$. We test that $489$ has order $2^2$ and that
$510$ has order $3^2$.  If it were not so, repeat the random choice
of $97$.

Note that $z=x^{27}$ has order $4$, and we solve inductively for $m(0),m(1)\in\{0,1\}$: the equation 
$y^{489^{2m(0)}}  = z^{-2}yz^{2}$ yields $m(0)=0$; now  $y^{489^{m(0)+2m(1)}}= y^{489^{2m(1)}}  = z^{-1}yz$ 
forces $m(1)=1$, and we have determined that $x^{-27}y x^{27}=y^{489^{0+2\cdot 1}}=y^{540}$.
        
Next we consider $z=x^4$ of order $27$, and we solve inductively for $m(0),m(1),m(2)\in \{0,1,2\}$: 
first, $z^{-9}yz^9= y^{510^{3^2m(0)}}$ yields $m(0)=0$. Now $z^{-3}yz=y^{510^{3(0+3m(1))}}$ yields 
$m(1)=1$. Lastly, $z^{-1}yz=y^{510^{0+3\cdot 1+3^2m(2)}}$ determines $m(2)=0$. Thus,  
$x^{-4} y x^4=y^{510^{3}}=y^{505}$.
 
  Note that  $1=\gcd(2^2,3^3)=2^2(7)+3^3(-1)$, and so
\begin{align*}
	x^{-1} yx = y^{505^7\cdot 540^{-1}\bmod{541}} = y^{316}.
\end{align*}  
Indeed, $\langle x,y\rangle\cong (\mathbb{Z}/{108})\ltimes_{\theta} (\mathbb{Z}/{541})$ where $\theta$ 
is defined by $1\theta=316$.  Note that $x^{54}$ centralizes $y$ and we still
recover the correct action, so we need not assume $x$ acts faithfully on $y$.  \hfill $\Box$
\end{example}

\subsection{Proof of  Theorem~\ref{thm:iso-meta-cyclic-BB}  (Isomorphism testing of coprime meta-cyclic groups)}
 
We need one further preliminary result.

\begin{lemma}\label{lem:meta-cyclic-iso}
Let  $A,\tilde{A},B,\tilde{B}$ be abelian groups such that the following diagram commutes:
\begin{align}
\xymatrix{ 
1 \ar[r] & B \ar[d]^{\beta}\ar[r] & A\ltimes_{\theta}B
	\ar[d]^{\gamma}\ar[r]
	& A \ar[d]^{\alpha}\ar[r] & 1\\
1 \ar[r] & \tilde{B} \ar[r] & \tilde{A}\ltimes_{\tilde{\theta}} \tilde{B}\ar[r] 
	& \tilde{A} \ar[r] & 1
}
\end{align}
Then $\gamma$ is an isomorphism if and only if $\alpha$ and $\beta$ are
isomorphisms and $\alpha\tilde{\theta}=\theta \Lambda_{\beta}$, where $\Lambda_{\beta}$
is conjugation by $\beta$. If $A$ and $B$ are cyclic, then $\gamma$ is an isomorphism if and only if $\alpha$ and $\beta$ are isomorphisms and $\theta$ and $\tilde\theta$  have the same image.
\end{lemma}
\begin{proof}
  Write elements of $A\ltimes_\theta B$ as $(a,b)$ with $a\in A$ and $b\in B$; analogously for the second group. Note that if $\gamma$ is an isomorphism as in the diagram, then it must map $B$ to $\tilde B$, hence
  \[(a,b)\gamma=(a\alpha,b\beta+a\psi)\]
  for isomorphisms $\alpha\colon A\to \tilde A$ and $\beta\colon B\to \tilde B$, and some map $\psi\colon A\to \tilde B$; note that in this case also $(a,b)\to (a\alpha,b\beta)$ describes an isomorphism, so we can assume that $a\psi=0$ for all $a\in A$. Now a direct calculation shows that $\gamma$ is an isomorphism if and only if $\alpha$ and $\beta$ are isomorphisms, and $\beta\cdot (a\alpha)\tilde\theta=a\theta\cdot \beta$ for all $a\in A$, which we write as  $\alpha\tilde\theta=\theta\Lambda_\beta$.  Now suppose  $A=\mathbb{Z}/a=\tilde{A}$ and $B=\mathbb{Z}/b=\tilde B$ are cyclic, so $\alpha,\beta,\theta,\tilde\theta$ are defined by the image of 1. In particular, $\Aut(B)$ is abelian and $\alpha\tilde\theta=\theta\Lambda_\beta$ becomes $\alpha\tilde\theta=\theta$. This shows that  if $\gamma$ is an isomorphism, then $\tilde\theta$ and $\theta$ have the same image. Conversely, suppose $\tilde\theta$ and $\theta$ have the same image, say $(1)\tilde\theta$ is multiplication by $\tilde \ell \in(\mathbb{Z}/b)^\times$ and $(1)\theta$ is multiplication by $\ell\in (\mathbb{Z}/b)^\times$, with $\langle \tilde \ell\rangle = \langle \ell\rangle$. The latter implies that $\tilde\ell^u\equiv \ell\bmod b$ for some $u$ coprime to the order of $\ell$ in $(\mathbb{Z}/b)^\times$. We need to find an automorphism $\alpha\colon A\to A$, determined by some $m\equiv (1)\alpha\bmod a$, such that $(1)(1\alpha\tilde\theta)\equiv (1)(1\theta)\bmod b$, that is, $\tilde\ell^m \equiv\ell\equiv\tilde \ell^u\bmod b$. If we have found such an $m$, then define $\beta\colon B\to B$ by $1\beta=1$, and observe that $\gamma\colon (a,b)\mapsto (a\alpha,b\beta)$ describes the required isomorphism.

  We now prove that one can find such an $m$. Write $a=cl$ where $l$ is the order of $\ell$ in $(\mathbb{Z}/b)^\times$. We want an integer $m$ with $\gcd(a,m)=1$ and $m\equiv u\bmod l$: then $1\mapsto m$ defines an automorphism of $A$ and $\tilde\ell^m\equiv \tilde \ell^u\bmod b$. Note that $\gcd(u,l)=1$ since $|\ell|=|\ell^u|$, but we might have $\gcd(a,u)\ne 1$ so we cannot choose $m=u$ in general. However, we can choose $m=u+xl$ where $x$ is the product of the prime divisors of $c=a/l$ which do not divide $u$: in this case clearly $m\equiv u\bmod l$, so $\tilde\ell^m\equiv \tilde\ell^u\bmod b$. It remains to show that $\gcd(m,a)=1$. If a prime $p$ divides $l$, then $p\nmid u$ since $\gcd(u,l)=1$, thus $p\nmid m$ since $m=u+xl$. If $p\mid c$, then there are two cases: first, if $p\mid u$, then $p\nmid l$ since $\gcd(u,l)=1$; by definition, $p\nmid x$, so $p\nmid m$ since $m=u+xl$. Second, if $p\nmid u$, then $p\mid x$ by construction, hence $p\nmid m$. This proves that $\gcd(m,a)=1$.
\end{proof}

\begin{proof}[Proof of Theorem~\ref{thm:iso-meta-cyclic-BB}]
By Lemma~\ref{lem:Count-3} we may use $\hat{\mathcal{D}}$ as our dense set of group
orders. Let $G$ be a group of known factored order $n\in \hat{\mathcal{D}}$.  Using 
Theorem~\ref{thm:meta-cyclic-recog}, we detect whether $G=U\ltimes K$ with both $U$ 
and $K$ cyclic of known orders.  This solves the first part of Theorem~\ref{thm:iso-meta-cyclic-BB}.
 
Next we must consider the isomorphism problem for this class of groups.
Let $G$ and $\tilde{G}$ be two such groups of order $n$, with decompositions $G=U\ltimes K$ and $\tilde{G}=\tilde{U}\ltimes \tilde{K}$ provided by Theorem~\ref{thm:meta-cyclic-recog}.  We also have two-way isomorphisms
$\alpha:\mathbb{Z}/c\to U$ and $\tilde{\alpha}:\mathbb{Z}/\tilde{c}\to \tilde{U}$,
and one-way isomorphisms $\beta:\mathbb{Z}/d\to K$ and 
$\tilde{\beta}:\mathbb{Z}/\tilde{d}\to \tilde{K}$;  we may assume $c=\tilde{c}$
and $d=\tilde{d}$ as these are isomorphism invariants. Using Theorem~\ref{thm:deconj} on inputs $x=1\alpha$ and $y=1\beta$ (respectively
$\tilde{x}=1\tilde{\alpha}$ and $\tilde{y}=1\tilde{\beta}$) we can construct 
$\theta\colon\mathbb{Z}/c\to \Aut(\mathbb{Z}/d)$ 
and $\tilde{\theta}\colon\mathbb{Z}/c\to \Aut(\mathbb{Z}/d)$ and one-way isomorphisms
\[\gamma\colon \mathbb{Z}/c\ltimes_{\theta} \mathbb{Z}/d\to G\quad\text{and}\quad \tilde{\gamma}\colon\mathbb{Z}/c\ltimes_{\tilde{\theta}} \mathbb{Z}/d\to \tilde{G}.\]
To decide isomorphism it is sufficient to show that the image $V$ of $\theta$ 
equals the image $\tilde{V}$ of $\tilde{\theta}$, see Lemma~\ref{lem:meta-cyclic-iso}.
Indeed, it suffices to show
$1\tilde{\theta}\in V$ and that $|V|=|1\tilde{\theta}|$.  
For that we note that both $V$ and $\tilde{V}$ are cyclic
$\pi(\leq \log\log n)$-groups, so Theorem~\ref{thm:iso-abelian-BB} determines 
their orders and equips $V$ with a constructive membership test, see Theorem~\ref{prop:abelian-member}.  
Hence we can decide if $G\cong \tilde{G}$, without providing isomorphisms
in either direction.
\end{proof}

\section{Closing remarks}



\subsection{Proof of Theorem~\ref{thm:iso-almostall-Cayley}}\label{sec14}
The brute-force isomorphism test between groups $G$ and $\tilde{G}$ assigns a 
minimal generating set $(g_1,\ldots,g_{d(G)})$ for $G$ to an arbitrary one 
$(\tilde{g}_1,\ldots,\tilde{g}_{d(G)})$ in $\tilde{G}$, and tests 
whether that assignment induces an isomorphism $G\to \tilde{G}$.  In more detail,
assume the groups are written as regular permutation representations. 
Using \cite{KL:quo}*{P10}, we can compute a presentation $\langle X\mid R\rangle$ 
for $G$.   All that is done in time $O(|G|(\log |G|)^c)$ for some  constant $c$.
 For each map $\tau\colon X\to \tilde{G}$ 
and each SLP $\sigma\in R$, test if $X\tau\sigma\equiv 1$ in $\tilde{G}$; that test takes time 
$O(|G|^{d(G)+1}(\log |G|)^c)\subset O(|G|^{d(G)+1+o(1)})$.  It follows from \cite{guralnick}*{Theorem A} 
that for every group $G$ of $k$-free order $n$ we have  $d(G)\leq k+1$. 
Thus, the complexity of brute-force isomorphism testing of groups of order $n$ is
\begin{align}\label{eq:complexity}
	O(n^{\mu(n)+2+o(1)}) \quad\text{where}\quad \mu(n)  = \max\{k : \textnormal{ $n$ is not $k$-free }\}.
\end{align}
Now fix a constant $k>0$. Restricting this isomorphism test to group orders $n$ with  $\mu(n)\leq k$, it runs in time polynomial in the input size for groups given by Cayley tables. By \cite{num}*{(2)}, 
the density of $k$-free integers tends towards $1/\zeta(k)$, which approaches $1$ as $k\to\infty$. 
Thus, for a given $\varepsilon>0$, choose $k$ large enough such that the density 
of $k$-free integers is at least $1-\varepsilon$. \hfill $\Box$

\subsection{Type theory makes promise problems into decision problems}\label{sec:Turing}
We have done the unusual step of working with Type Theory, where most preceding theory
has developed in terms of Turing Machines.  There are several good reasons for this, 
including how accurately Type Theory models contemporary programming.
Most recent programming languages, such as C\#, Haskell,  ML languages, and Scala (Dotty) 
have had their Type Theory machine verified, thus offering some confidence that 
calculations in these languages can be trusted -- a claim that traditional languages 
like Fortran and C cannot provide.  Such verification should seem important to the 
mathematical community and encourage greater adoption of Type Theory models.  
Still, there is a further more important obstacle we have faced when attempting to use conventional
Turing models; in the end, it was that obstacle that forced our deviation.

\subsubsection*{Black-box groups are not all groups of black-box type.} 
Observe that every group of black-box type is a black-box group in the sense of Babai \& Szemered\'i,
see \cite{Seress}*{Chapter~2}, but strictly speaking the converse may not hold.  That
is because a black-box group is a model of a Turing Machine in which groups are
input by strings of uniform (or bounded) length, and oracles perform the group operations
at unit cost.  However, not all strings of input of the appropriate length are required
to describe a group. So in principle a black-box group problem asks the harder question 
of computing with groups of unknown types, \emph{as well as rejecting mal-formed inputs}.
This does not only mean to reject unparseable strings -- but to detect correctly that
reasonable inputs are not encoding something weaker than a group, such as a proper semigroup or 
a nonassociative quasigroup.  If an algorithm for a black-box group merely assumes
the input is a group, then there is no theoretical foundation to guarantee the correctness of the outcome.
All black-box group algorithms must somehow distinguish groups from facsimiles.

In formal terms, the Babai \& Szemered\'i model of black-box groups forces many problems
about these groups to be \emph{promise problems}, that is, partitions of the set of all
strings $\Sigma^*$ over an alphabet $\Sigma$ into strings that are \emph{accepted}, 
\emph{rejected}, and \emph{ignored}. The usual \emph{decision problems} have no ignored 
strings and those are the partitions of $\Sigma^*$ that are customary in the definitions 
of P and NP. For a detailed account  of the difference between such problems see \cite{promise}.  
The promise problem perspective is a perfectly appropriate and relevant approach 
to the study of groups,  but we shall like to be clear about the difference and 
explain we are working explicitly in the decision problem side.

\subsubsection*{Strategies for rejecting  mal-formed inputs.} The literature is 
concerned with the difference of promise and decision problems and offers several 
strategies, see for example \cite{promise} and bibliography contained therein.  For instance, 
for problems in NP, a certificate can be provided to test whether the conclusion is independent 
of the computation.  Unfortunately, black-box group isomorphism is not known to be 
in NP or co-NP \cite{BS84}, certificates do not yet exist. Another approach says
if an algorithm is known to run in time $T(n)$ on well-formed inputs of length $n$, 
then an input of length $n$ that requires more steps must be mal-formed.  On a practical
level, bounds on $T(n)$ are usually asymptotic and thus not immediately useful in predicting
for a specific $n$ when $T(n)$ has been exceeded.  Even if $T(n)$ is carefully modeled,
this does not prevent a mal-formed input tricking an algorithm into accepting the 
string in time $T(n)$, which means the algorithm accepts the wrong language. 
Another approach is to assume the operation oracles distinguish well-formed
inputs, for example, by rejecting inputs that are not in the hypothetical group.  Yet 
that assumption can have the effect of requiring that the oracles provide a 
membership test for the group.  Membership tests are usually a goal for a black-box group algorithm, not
a starting assumption.

\subsubsection*{Type theory avoids mal-formed inputs.} Since our focus is on isomorphism testing, we cannot apply the above remedies. So we
have appealed to a computational model without those issues.
In a Type Theory model, all inputs are terms of a known type.  In particular,
a group $G\colon\code{Group[A]}$ comes with operations, generators, and, crucially, certificates 
of the axioms of a group.  Because of this \emph{it is impossible 
to provide inhabitants of \code{Group[A]} that
are not groups!}  Hence, within our proofs of the various algorithms, we can safely make
that assumption about our inputs.   

\subsubsection*{Reconciliation.} Even so, we wish to emphasize that many, if not all,  
``black-box group algorithms'' make at least implicit assumptions that all inputs are guaranteed
to define groups.  Some authors adopt language such as ``grey box groups'' or provide 
surrounding discussions hypothesizing a context such that the elements are drawn from some 
larger class of group which is not black-box, such as $\mathrm{GL}_d(\mathbb{F}_q)$ or 
$\mathrm{Sym}(\Omega)$, and therefore the resulting input can be trusted to be a group; 
see for example the discussions in \cite{handbook}*{Section~3.1.4}.  
Such examples are in fact equivalent to discussing algorithms for groups of black-box type, 
rather than strict black-box group algorithms.  So with Type Theory we are not doing more 
than what is done elsewhere, we are simply clarifying the model so that it can be
properly assessed within the polynomial-time hierarchy. 

\subsubsection*{By-products.} Independent of our technical concerns, the adoption of types has proved fruitful in 
proving stronger theorems than expected.  For example, the norm for decades has been to describe group
membership test as being relative to some larger group, for example, given an subgroup 
$H\leq \mathrm{Sym}(\Omega)$ and $g\in \mathrm{Sym}(\Omega)$, decide if $g\in H$, see \cite{Seress}*{p.~56}.
Notice, however, that our definition of {\em constructive membership} (see Section \ref{secMS}) decides 
membership for completely {\em arbitrary} input $x:\code{A}$, whether or not $x$ lies
in some larger group.  So we move from the promise problem hierarchy to the decision problem 
hierarchy.  Arguably this is the behavior one intuitively expects by a function
that claims to decide membership in a set, and it is the kind that can be programmed into a
computer leaving no input with unspecified behavior.  Our form
of absolute membership testing was attained in Proposition~\ref{prop:abelian-member} and 
Proposition~\ref{prop:abelian-pres} as a direct result of appealing strictly to types.

{\small
\bibliographystyle{abbrv}

}

\end{document}